\documentclass[]{amsart}

\usepackage{amsmath}
\usepackage{amsthm}
\usepackage{amssymb,amsfonts}
\usepackage{mathtools}
\usepackage{color}

\usepackage[all,arc]{xy}
\usepackage{tikz}
\usetikzlibrary{shapes,decorations,arrows, chains, positioning, shapes.geometric, shapes.symbols,calc}

\def\centerarc[#1](#2)(#3:#4:#5)
    { \draw[#1] ($(#2)+({#5*cos(#3)},{#5*sin(#3)})$) arc (#3:#4:#5); }

\newtheorem{theorem}{Theorem}[section]

\newtheorem{corollary}[theorem]{Corollary}

\newtheorem{lemma}[theorem]{Lemma}
\newtheorem{proposition}[theorem]{Proposition}
\newtheorem{remark}[theorem]{Remark}

\numberwithin{equation}{section}

\newcommand{\C}{\mathbf{C}}
\newcommand{\R}{\mathbf{R}}
\newcommand{\dd}{\partial \bar{\partial}}
\newcommand{\p}{\partial}

\newcommand{\tu}{\tilde{u}}
\newcommand{\vol}{\mbox{Vol}}

\begin{document}
	
\title[Schauder estimates]{Schauder estimates on products of  cones}
\author[M. de Borbon]{Martin de Borbon}
\address{LMJL, Universit\'e de Nantes }
\email{martin.deborbon@univ-nantes.fr}
\author[G. Edwards]{Gregory Edwards}
\address{University of Notre Dame}
\email{gedward2@nd.edu}	
		
\maketitle


\maketitle

\begin{abstract}
	We prove an interior Schauder estimate for the Laplacian on metric products of two dimensional cones with a Euclidean factor, generalizing the work of Donaldson and reproving the Schauder estimate of Guo-Song. We characterize the space of homogeneous subquadratic harmonic functions on products of cones, and identify scales at which geodesic balls can be well approximated by balls centered at the apex of an appropriate model cone. We then locally approximate solutions by subquadratic harmonic functions at these scales to measure the H\"older continuity of second derivatives.
\end{abstract}

\section{Introduction}

In \cite{Donaldson}, Donaldson laid the foundation for studying geometric and analytic properties of K\"ahler metrics with conical singularities along a smooth hypersurface. There, Donaldson proved the Schauder estimate for the Laplacian of such metrics, establishing an important step toward the eventual solution of the Yau-Tian-Donaldson conjecture relating the existence of K\"ahler-Einstein metrics on Fano manifolds to K-stability \cite{CDS15a,CDS15b,CDS15c}. The Schauder estimate for metrics with cone singularities along a smooth hypersurface was later reproved by Guo-Song \cite{GuoSongI} and Gui-Yin \cite{GuiYin}.

Conical K\"ahler metrics have since received considerable attention by geometers, particularly for K\"ahler metrics having conical singularities along a simple normal crossing divisor. The Schauder estimate is an important step in developing the linear elliptic theory necessary to construct conical K\"ahler-Einstein metrics, and Guo-Song established the Schauder estimate for linear elliptic and parabolic equations with conical singularities along a simple normal crossing divisor \cite{GuoSongII}.

Near an intersection of irreducible components of the cone divisor, these conical K\"ahler metrics are locally modeled by the Cartesian product of two dimensional cones with cone angle less than $2\pi$, and a Euclidean factor. In this paper we consider such products of cones as themselves Riemannian cones with singular rays, and establish a sharp Schauder estimate for the Laplacian, generalizing the work of Donaldson and reproving the results of Guo-Song. Our Schauder estimate is independent of Guo-Song, and we present it as the techniques may prove useful in further works. 

\subsection{Main result}
Let us write \(C(S^1_{2\pi \beta})\) for the cone over a circle of length \(2\pi\beta >0\). For \(m \geq 3\) and \(n \geq 1\), we consider the product space
\[ C(S^1_{2\pi\beta_1}) \times \ldots \times  C(S^1_{2\pi\beta_n}) \times \R^{m-2n} . \]
In this article we restrict to the case where \(0 < \beta_a < 1\) for all \(a= 1, \ldots, n\). We use polar coordinates \((r_a, \theta_a)\) on each of the cone factors and standard Cartesian coordinates \(s = (s_1, \ldots, s_{m-2n})\) on the Euclidean part, so the cone metric in these coordinates is
\[g = \sum_{a=1}^{n} \left(dr_a^2 + \beta_a^2 r_a^2 d\theta_a^2\right) + ds^2 .\]

We use \(a, b\) to denote indices in \(1, \ldots, n\), corresponding to the cone factors, and indices \(i, j\)  in \(1, \ldots, m-2n\) for the Euclidean directions.
Let \(\mathcal{D}_2\) be the set of second order differential operators given by:
\begin{itemize}
	\item pure Euclidean: 
	\[\frac{\p^2}{\p s_i\p s_j} ;\]
	\item mixed conical-Euclidean: 
	\[\frac{\p^2}{\p{r_a}\p s_i}, \hspace{2mm} \frac{1}{r_a} \frac{\p^2}{\p \theta_a \p s_i} ;\]
	\item mixed conical, \(a \neq b\):
	\[\frac{\p^2}{\p{r_a} \p r_b}, \hspace{2mm} \frac{1}{r_a r_b}\frac{\p^2}{\p \theta_a \p \theta_b}, \hspace{2mm} \frac{1}{r_b} \frac{\p^2}{\p r_a \p\theta_b} ;\]
	\item conical Laplacians:
	\[\Delta_{\beta_a} = \frac{\p^2}{\p r_a^2} + \frac{1}{r_a} \frac{\p}{\p r_a} + \frac{1}{\beta_a^2 r_a^2} \frac{\p^2}{ \p\theta_a^2} . \] 
	\item pure conical, for indices \(a\) with \(\beta_a < 1/2\):
	\[\frac{\p^2}{\p{r_a}^2}, \hspace{2mm} \frac{1}{r^2_a}\frac{\p^2}{\p \theta_a^2}, \hspace{2mm} \frac{1}{r_a} \frac{\p^2}{\p r_a \p\theta_a} ;\]
\end{itemize}
In this collection of differential operators we are excluding the pure conical elements of the Hessian for those angles $\beta\geq 1/2$.

Our main result is a version of the interior Schauder estimates for the Laplace operator of our singular metric,
\[\Delta = \Delta_{\beta_1} + \ldots + \Delta_{\beta_n} + \Delta_{\R^{m-2n}} .\]

We define
\begin{equation}
	\mu = \min \begin{cases} 1 \\
	1/\beta_a -1 \hspace{2mm} \mbox{for } \beta_a \geq 1/2 \\
	1/\beta_a -2 \hspace{2mm} \mbox{for } \beta_a < 1/2. \\
	\end{cases}
\end{equation}

\begin{theorem} \label{mainthm}
	Let \(0< \alpha < \mu\). There is a constant \(C= C(\alpha, \beta_a, m)\) with the following property.
	If \(u \in W^{1, 2}(B_2)\) is a weak solution of \(\Delta u =f\) with \(f \in C^{\alpha}(B_2)\), then for every \(x \in B_1\) and every \(D \in \mathcal{D}_2\) there is \(\tau \in \R\) and constant \(K>0\) such that 
	\begin{equation}\label{mainestimate}
		\left(\rho^{-m} \int_{B(x, \rho)} |Du - \tau|^2 \right)^{1/2} \leq K \rho^{\alpha}   
	\end{equation}
	for every \(0< \rho <1\). Moreover,
	\begin{equation} \label{mainestimate2}
		|\tau| + K \leq C \left(|f(x)| + |f|_{C^{\alpha}(x)} + \|u\|_{L^2(B_2)} \right) .
	\end{equation}
\end{theorem}

Our main estimate \eqref{mainestimate} gives a bound for \(Du\) in a Campanato space, which is known to be equivalent to a bound on its H\"older coefficient:
\begin{equation*}
\sup_{|x-y|<1/2} \frac{|Du(y)-Du(x)|}{|x-y|^{\alpha}} .
\end{equation*}
We point out three main differences between Theorem \ref{mainthm}
and the standard Schauder estimates for the Euclidean Laplacian.

\begin{enumerate}
	
	\item We do not estimate all the second order derivatives.  Indeed, it is easy to see that 
	\[u = r^{1/\beta} \cos \theta \]
	is a harmonic function on \(C(S^1_{2\pi\beta})\) and \(\p_r^2 u \sim r^{1/\beta -2}\) is unbounded if \(\beta>1/2\) and \(r \to 0\).
	Note however that \(\mathcal{D}_2\) contains all entries of the Hessian except for the terms \(\p_{r_a}^2\) and \(r_a^{-2}\p_{\theta_a}^2\) and \(r_a^{-1}\p_{\theta_a r_a}^2\) for those \(\beta_a > 1/2\).

	\item We require \(\alpha<1/\beta -1\) for those \(\beta>1/2\). To see that this upper bound on \(\alpha\) is optimal, consider the harmonic function
	\[u = r^{1/\beta} s \cos \theta \]
	on \(C(S^1_{2\pi\beta}) \times \R\). Then \(\p^2 u/\p r\p s\) is \(C^{\alpha}\) for \(\alpha=1/\beta-1\). Similarly, for \(\beta<1/2\), we have that \(\p^2_{r}(r^{1/\beta} \cos \theta)\) is \(C^{\alpha}\) with \(\alpha=1/\beta-2\).
	
	A consequence of our methods is that the upper bound on \(\alpha\) is determined by the next indicial root after \(2\). That is, there are no homogeneous harmonic functions with growth rate \(2+ \alpha\) for any \(0 < \alpha < \mu\). Note that, since we always assume \(0< \alpha<1\), the restriction \(\alpha< 1/\beta_a -2\) only applies if \(1/3<\beta_a<1/2\).
	
	\item  For \(u\) as in Theorem \ref{mainthm}, the first derivatives \(\p_{r_a} u\) and \(r_a^{-1}\p_{\theta_a}u\) vanish along \(\{r_a=0\}\). This is a manifestation of the fact that \(g\) has non-trivial holonomy along arbitrary small loops that go around the conical set. In particular, \(\p_{r_a} P\) and \(r_a^{-1}\p_{\theta_a}P\) vanish identically for any `sub-linear harmonic polynomials' \(P \in \mathcal{H}_{\leq 1}\); see Proposition \ref{subquadraticprop}.

\end{enumerate}

Our metric is locally Euclidean on its regular part, so it follows from the standard regularity theory for the Euclidean Laplacian that \(u\) is locally \(C^{2, \alpha}\) on the locus where none of the \(r_a\) vanishes. The second derivative \(Du\) is then point-wise defined outside the conical singularities, and Theorem \ref{mainthm} guarantees that \(Du\) extends continuously over points \(x\) lying on the singular part \(S=\cup_a \{r_a=0\}\) by setting \(Du(x) = \tau\). Note however that the estimate \eqref{mainestimate} is non-trivial even when \(x\) lies on the regular part. In fact, for regular points arbitrary close to the singular set, the Euclidean Schauder estimates only apply in arbitrary small scales, while our bound \eqref{mainestimate} holds for all \(0 < \rho < 1\) independent of \(x\). 

\subsection{Sketch of proof of Theorem \ref{mainthm}}

We use the approximation by polynomials technique, see {\cite[Chapter 5.4]{HanLin}}. For simplicity, let us assume that \(u\) is harmonic. The idea is then to expand \(u\) on \(B(x, 1)\) as a sum of `homogeneous harmonic polynomials' and then  subtract the subquadratic part \(P\) from \(u\), so that \(u-P\) will vanish to order \(2+\alpha\) at \(x\). If we set \(\tau=DP(x)\), then \(Du-\tau\) will vanish to order \(\alpha\)
as stated in equation \eqref{mainestimate},  provided we have control on the \(C^{\alpha}\)-norm of \(DP\). The problem is that \(x\) is not necessarily the vertex of a cone, so there are no dilations centered at \(x\) and there is not such a notion as `homogeneous harmonic polynomials' on $B(x,1)$. 

To deal with this we use appropriate model cones at each scale. More precisely, we fix some \(0<\lambda<1\) and consider the balls \(B(x, \lambda^k)\) for integers $k\geq 0$. The key point is that, if we fix some error \(\epsilon_0>0\), then all but at most a finite number \(N(\epsilon_0, \lambda)\) of the scales \(\lambda^k\) are `good', meaning that the rescaled ball \(\lambda^{-k}B(x, \lambda^k)\) is isometric to a unit ball in a model cone \(C(Y)\)
with center at distance at most \(\epsilon_0\) from the vertex; see Section \ref{saclesect}.
We approximate \(u\) by subquadratic harmonic functions at each scale \(\lambda^k\) in an iterative manner. We start with \(u_0 =u\) on \(B(x, 1)\). If \(\lambda^k\) is a good scale,  we set \(u_{k+1}=u_k-P_k\) on \(B(x, \lambda^{k+1})\) where \(P_k\) is the subquadratic part; see Proposition \ref{keystep}. If \(\lambda^k\) is a bad scale, we let \(u_{k+1}=u_k\). We set \(\tau_k = DP_k(x)\) and verify that \(\tau= \sum_k \tau_k \) approximates \(Du\) around \(x\) in the sense that the estimate \eqref{mainestimate} holds.

Note that we do \emph{not} obtain a harmonic `polynomial \(P= \sum_k P_k\)' on \(B(x, 1)\) that approximates \(u\) up to order \(2+\alpha\). Indeed, the polynomials \(P_k\) are defined on smaller scales as \(\lambda^k \to 0\). However, their values \(DP_k(x)\) are defined and we can identify `\(\tau=DP(x)\)'.

\subsection{Angles equal to \(\pi\)}\label{sec:pi}

The case where some cone factors have angle $\pi$ is somewhat exceptional due to the additional quadratic harmonic function, $r_a^2 e^{i\theta_a}$.

Assume \(\beta_a=1/2\) for some cone factor and let \(u \in W^{1,2}\) be a weak solution of \(\Delta u = f\) with \(f \in C^{\alpha}\). We use the branched covering map \((r_a, \tilde{\theta}_a) \to (r_a, 2\tilde{\theta}_a)\) to pull-back \(C(S^1_{2\pi\beta_a})\) into \(\R^2\). Let \(\tilde{u}\) and \(\tilde{f}\) be the pull-backs of \(u\) and \(f\). It is easy to see that \(\tilde{\Delta} \tilde{u} = \tilde{f}\), by writing test functions as a sum of even an odd parts with respect to \(\tilde{\theta} \to \tilde{\theta} + \pi\). Then, in addition to the results of Theorem \ref{mainthm}, we have
\[
    \big\| \frac {\partial^2 u} {\partial r_a^2} \big\|_{L^\infty(B_1)} +  \big\| \frac 1 {r_a} \frac {\partial u} {\partial r_a} \big\|_{L^\infty(B_1)} +  \big\| \frac 1 {r_a^2} \frac{\partial^2 u} {\partial \theta_a^2} \big\|_{L^\infty(B_1)} \leq C,
\]
so that the pure conical derivatives in the \(C(S^1_{2\pi\beta_a})\) direction are bounded for every index $a$ with $\beta_a = 1/2$. Note however that this bound is sharp as the second derivative \(\p_r^2\) of the harmonic function \(r_a^{2}e^{i\theta_a}\) does not extend continuously over \(r_a=0\).

\subsection{Comparison with other works}

A first precedent for Theorem \ref{mainthm} is the work of Donaldson \cite{Donaldson}, which considers the case \(C(S^1_{2\pi\beta}) \times \R^{m-2}\). Donaldson proves the Schauder estimates via classical potential theory, differentiating twice Green's representation formula. The main work is on deriving a suitable `polyhomogeneous' expansion of the Green's function around the conical set, which is done via separation of variables. Later, Guo-Song \cite{GuoSongI} gave a new proof of Donaldson's Schauder estimate without using potential theory. The method of Guo-Song gives a sharp estimate on the modulus of continuity of the second derivatives in terms of the Dini condition, along the lines of Wang \cite{Wang} in the smooth case. Their proof relies on approximation by smooth metrics with non-negative Ricci curvature and the Cheng-Yau gradient estimate. More recently, Guo-Song \cite{GuoSongII} adapted their method to the case of normal crossing singularities  \(C(S^1_{2\pi\beta_1}) \times \ldots \times  C(S^1_{2\pi\beta_n}) \times \R^{m-2n}\) as considered in this paper. Yet another approach for the Schauder estimates on \(C(S^1_{2\pi\beta}) \times \R^{m-2}\) was given by Gui-Yin \cite{GuiYin}, considering also \(\beta>1\) and higher order estimates. A main ingredient in Gui-Yin's work is an expansion formula for bounded harmonic functions, \cite[Proposition 5.2]{GuiYin}, and their methods are closer to ours.

The main feature of our work, is that we exploit the explicit knowledge of \(\mathcal{H}_{\leq 2}\), the space of homogeneous harmonic functions of degree \(\leq 2\). We derive Theorem \ref{mainthm} from the bounds on \(|DP|_{\alpha}\) for  \(D \in\mathcal{D}_2\) and \(P \in \mathcal{H}_{\leq 2}\). Indeed \(|DP|_{\alpha}=0\), see Proposition \ref{subquadraticprop} and Corollary \ref{subquadcor}. There are other variants of Theorem \ref{mainthm} that can be obtained by minor changes to our proof. For example, we can relax the requirement \(\alpha<1/\beta_a -2\) for those angles \(\beta_a<1/2\), if we are not willing to estimate \(\p^2/\p r_a^2\), etc. Our work does not rely on approximation by smooth metrics, and we expect to extend it to more general conical type singularities.

\subsection*{Organization}
The rest of the paper is organized as follows. In Section 2 we establish elementary properties of weak solutions and harmonic approximation, see Lemma \ref{harmaproxlem}. In Section 3 we compute the subspace of homogeneous subquadratic harmonic functions for the model cones, see Proposition \ref{subquadraticprop}. In Section 4 we bound the number of bad scales (Lemma \ref{scaleslemma}), and establish the monotonicity property for harmonic functions (Lemma \ref{monlemma}). Finally, in Section 5 we complete the proof of Theorem \ref{mainthm}.

\subsection*{Acknowledgments}
The authors would like to thank G\'abor Sz\'ekelyhidi for suggesting this approach to prove the Schauder estimates.

The first named author was  supported by the ANR grant \textbf{ANR-17-CE40-0034}: \emph{CCEM} and would like to thank Gilles Carron for useful discussions.

The second author was supported by the National Science Foundation RTG: Geometry and Topology at the University of Notre Dame, grant number \textbf{DMS-1547292}.

\section{\(L^2\)-theory} \label{L2}

\subsection{Preliminaries and notation}
We  write
\begin{equation*}
	\R^m_{(\beta)} =  C(S^1_{2\pi\beta_1}) \times \ldots \times  C(S^1_{2\pi\beta_n}) \times \R^{m-2n} 
\end{equation*}
for our product space. We denote by either \(B(x, \rho) \) or \(B_{\rho}(x)\), a geodesic ball of center \(x\) and radius \(\rho\), with respect to the distance \(d=d_g\) induced by \(g\). If the center of the ball is zero, we simply write \(B_{\rho} = B_{\rho}(0)\).

Note that \(g\) is uniformly equivalent to the Euclidean metric, since 
\begin{equation} \label{unifequiv}
	(\min_a\beta_a)^2 g_{\R^m} \leq g \leq g_{\R^m} .
\end{equation}
In particular, \((\min_a\beta_a) d_{\R^m} \leq d \leq d_{\R^m}\) and the geodesic balls of \(g\) are uniformly comparable to Euclidean balls, in the sense that  
\[B_{\R^m}(x, \rho) \subset B(x, \rho) \subset B_{\R^m}(x, (\min_a\beta_{a})^{-1} \rho) .\] 
Note that, if the center of the ball is located at zero, then \(B_{\rho} = \{\sum_a r_a^2 + \sum_i s_i^2 < \rho^2\} \) agrees with the standard Euclidean ball.

Our \(W^{1, 2}\) and \(C^{\alpha}\) function spaces are the standard ones of \(\R^m\). These Sobolev and H\"older spaces can also be interpreted by measuring the corresponding defining norms with respect to the conically singular metric \(g\). 
If follows from equation \eqref{unifequiv} that the respective \(W^{1,2}\) and \(C^{\alpha}\) norms defined by \(g\) and \(g_{\R^m}\) are then equivalent. 
Likewise, our integrals are always referred  with respect to the Riemannian volume measure \(dV_g\). Still, \(dV_g\) is a constant multiple of the ordinary Lebesgue measure; the constant factor being equal to the product of the \(\beta_a\). 

\begin{remark}
	Most of the results in Sections \ref{L2}, \ref{Hom} and \ref{Quant} are well known in the much more general settings of stratified metrics, see \cite{ACMII, ACM}, and Ricci limit spaces, see \cite{Ding}. Our space \(\R^m_{(\beta)}\) fits into both of these theories. Since proofs are elementary, we have decided to include them. 
\end{remark}

\subsection{The Sobolev space \(W^{1, 2}\)}

The gradient of a function \(u\) that is \(C^1\) on the regular set is given by the following expression
\[\nabla u = \sum_a \left( \frac{\p u}{\p{r_a}} \p_{r_a} + \frac{1}{\beta_a r_a}\frac{\p u}{\p \theta_a}  \frac{1}{\beta_a r_a} \p_{\theta_a} \right) + \sum_i \frac{\p u}{\p{s_i}} \p_{s_i} . \]
For a domain \(\Omega \subset \R^m_{(\beta)}\) we set \(H^1(\Omega)\) to be the standard Sobolev space, given by the completion of Lipschitz functions under the norm
\begin{equation*} \label{W12norm}
\| u \|^2_{W^{1, 2} (\Omega)} = 	\| u \|_{L^{2} (\Omega)}^2 + 	\| \nabla u \|_{L^{2} (\Omega)}^2 .
\end{equation*}
For regular domains \(\Omega\) (e.g. a ball) \(H^1(\Omega)\) agrees with \(W^{1, 2}(\Omega)\), that is the subspace of functions in \(L^2(\Omega)\) that have first order weak derivatives also in \(L^2(\Omega)\). We will always work with regular domains and keep the notation \(W^{1, 2}\). 

It is a standard fact that, on a space with co-dimension two singularities,  smooth functions  with compact support on the regular locus, \(C^{\infty}_\mathrm{c}(\Omega^{\mathrm{reg}})\), are dense in \(W^{1, 2}(\Omega)\); see \cite{Mondello}. In our case, this follows by taking products with the following cut-off functions:
\begin{lemma}
	For any \(\epsilon>0\) there is a Lipschitz function \(\chi\) on \(\R^2\) with the following properties:
	\begin{itemize}
		\item \(\chi = 0\) in a neighborhood of \(0\);
		\item \(\chi = 1\) outside \(B_{\epsilon}\);
		\item \(\int_{\R^2} |\nabla \chi|^2 < \epsilon\) .
	\end{itemize}
\end{lemma}

\begin{proof}
	Let \((r, \theta)\) be polar coordinates in \(\R^2\). For \(\delta, Q >0\) we set
	\begin{equation*}
	f_{\delta, Q} = \begin{cases}
	0 \hspace{2mm} \mbox{if } r \leq \delta \\
	\log (r/\delta) \hspace{2mm} \mbox{if } \delta \leq r  \leq Q\delta \\
	\log Q \hspace{2mm} \mbox{if } r \geq Q\delta .
	\end{cases}
	\end{equation*}
	Let \(1- \chi = (\log Q)^{-1} \left( \log Q -  f_{\delta, Q} \right)\), so
	\begin{align*}
	\int_{\R^2} |\nabla \chi|^2 &= (\log Q)^{-2} \int_{\delta}^{Q\delta} \frac{1}{r^2} r dr \\
	&= (\log Q)^{-1} .
	\end{align*}
	We take \(Q>>1\) such that \((\log Q)^{-1} < \epsilon\) and then we choose \(\delta<<1\) such that \(Q \delta < \epsilon\).
\end{proof}

\subsection{Weak solutions}

Let \(u \in W^{1,2}(\Omega)\) and \(f \in L^1_{\mathrm{loc}}(\Omega)\). We say that \(-\Delta u = f\), if for every Lipschitz test function \(\psi\) with compact support contained in \(\Omega\), we have
\begin{equation}\label{weakeq}
\int_{\Omega} \langle \nabla u, \nabla \psi \rangle =  \int_{\Omega} f \psi .
\end{equation}
The space \(W^{1, 2}_0(\Omega) \subset W^{1, 2}(\Omega)\) is the completion of compactly  supported Lipschitz functions. By continuity, the identity \eqref{weakeq} also holds for test functions \(\psi \in W^{1, 2}_0(\Omega)\). 

It is straightforward to show existence of weak solutions, by using the variational method as follows:
The Poincar\'e inequality asserts that, if \(\Omega\) is bounded, 
\[\int_{\Omega} u^2 \leq C \int_{\Omega} |\nabla u|^2 \]
for \(u \in W^{1, 2}_0(\Omega)\). As a consequence, given \(f \in L^2\), the functional \(\varphi \mapsto \int \varphi f\) is continuous with respect to the Dirichlet norm \(\| \varphi \|_D^2 = \int |\nabla \varphi|^2\). By Riesz, there is \(u \in W^{1, 2}_0\) such that \(\langle \cdot, u \rangle_D = \int \cdot f\), which is to say that \(u \in W^{1, 2}_0\) is a weak solution of \(-\Delta u = f\).

\begin{lemma}[Caccioppoli inequality] \label{Caccioppoli}
	Let \(u \in W^{1, 2}(B_1)\) solve \(-\Delta u = f\) with \(f \in L^2(B_1)\), then
	\begin{equation*}
	\int_{B_{1/2}} |\nabla u|^2 \leq  \int_{B_1} f^2 + C \int_{B_1} u^2 .
	\end{equation*}
\end{lemma}

\begin{proof}
	Let \(\eta\) be a compactly supported function in \(B_1\), with \(\eta = 1\) on \(B_{1/2}\). We set our test function to be \(\psi = \eta^2 u \in W^{1, 2}_0(B_1)\), so
	\begin{align*}
	\int_{B_1} \eta^2 |\nabla u|^2 &= \int_{B_1} \eta^2 u f - 2 \int_{B_1} \eta u \langle \nabla u, \nabla \eta \rangle \\
	&\leq \frac{1}{2} \int_{B_1} \eta^2 u^2 + \frac{1}{2} \int_{B_1} \eta^2 f^2  + \frac{1}{2} \int_{B_1} \eta^2 |\nabla u|^2  + 2 \int_{B_1} u^2 |\nabla \eta|^2 .
	\end{align*}
	The statement follows with \(C = \max \left(\eta^2 + 4 |\nabla \eta|^2 \right) \)
\end{proof}

The Caccioppoli inequality, combined with Rellich's compactness theorem give us the following useful result.

\begin{lemma}[Harmonic approximation] \label{harmaproxlem}
	For every \(\epsilon>0\) there is \(\delta>0\) with the following property: If \(u \in W^{1, 2}(B_1)\) satisfies \(\Delta u = f\) with \(\|u\|_{L^2(B_1)} \leq 1\) and \(\|f\|_{L^2(B_1)} < \delta\), then there is a weak harmonic function \(h \in W^{1, 2}(B_{1/2})\) such that \(\|u-h\|_{L^2(B_{1/2})} < \epsilon\).
\end{lemma}

\begin{proof}
	If not, we would have some \(\epsilon_0>0\) and a sequence of functions \(u_i\) with \(\|u_i\|_{L^2(B_1)} =1\), \(\| \Delta u_i\|_{L^2(B_1)} \to 0 \) and \(\|u_i -h\|_{L^2(B_{1/2})} \geq \epsilon_0\) for every weak harmonic function \(h\) on \(B_{1/2}\). The Caccioppoli inequality implies that \(\|u\|_{W^{1, 2}(B_{1/2})} \leq C\). The Rellich theorem asserts that the inclusion
	\[W^{1, 2} (B_{1/2}) \subset L^2(B_{1/2}) \]
	is compact. 
	Taking a subsequence, we have \(u_{\infty} \in W^{1, 2}(B_{1/2})\) with \(\lim_{i \to \infty} \|u_i - u_{\infty}\|_{L^2(B_{1/2})} = 0\) and \(u_i\) converges weakly in \(W^{1, 2}\) to \(u_{\infty}\). For every test function \(\psi\) we have 
	\begin{equation*}
	\int_{B_{1/2}} \langle \nabla u_{\infty}, \nabla \psi \rangle =  \lim_{i \to \infty} \int_{B_{1/2}} \langle \nabla u_{i}, \nabla \psi \rangle = 0 ,
	\end{equation*}
	so \(u_{\infty}\) is weakly harmonic. Taking \(h=u_{\infty}\) gives \(\|u_i -u_{\infty}\|_{L^2(B_{1/2})} \geq \epsilon_0\),  a contradiction. 
\end{proof}

\subsection{Interior \(L^2\)-bound on \(\mbox{Hess}(u)\)}

The Hessian of a function with respect to our singular metric, \(\mbox{Hess}(u) = \nabla_g^2 u\),  is uniformly equivalent, as a quadratic form, to the Euclidean Hessian. Same as before, we have the usual Sobolev space \(W^{2, 2}\), and it is irrelevant whether we use the Euclidean metric or \(g\) to define its norm. However, one difference is that smooth, compactly supported functions on the regular locus are no longer dense in \(W^{2, 2}\). The next result relies on the fact that eigenfunctions of the Laplace operator are Lipschitz, see \cite{Mondello} and \cite[Proposition 5.4]{BKMR}, which depends on our cone angles \(\beta_a\) being \(< 1\) for all \(a\).

\begin{proposition}\label{L2bound}
	Let \(u \in W^{1, 2}(B_1)\) with \(-\Delta u =f\) and \(f \in L^2(B_1)\). Then \(u \in W^{2, 2}(B_{1/2})\) and
	\begin{equation} \label{hessbound}
		\|\mathrm{Hess}(u)\|_{L^2(B_{1/2})} \leq C \left(\|f\|_{L^2(B_1)} + \|u\|_{L^2(B_1)} \right) 
	\end{equation}
\end{proposition}

\begin{proof}

Let \(\eta\) be a smooth cut-off function, with compact support in \(B_1\) and identically equal to \(1\) on \(B_{1/2}\). We can clearly arrange so that \(\Delta \eta\) is also smooth. The function \(\tilde{u} = \eta u \in W^{1, 2}_0(B_1)\)  satisfies \(-\Delta \tilde{u} = \tilde{f}\) with \(\tilde{f} = -(\Delta \eta) u - 2\langle \nabla u, \nabla \eta \rangle + \eta f \). Applying the Caccioppoli inequality, we obtain
\[ \|\tilde{f}\|_{L^2(B_1)} \leq C \left(\|f\|_{L^2(B_1)} + \|u\|_{L^2(B_1)} \right) .\] 
We shall show that \(\tilde{u} \in W^{2, 2}(B_1)\cap W^{1, 2}_0(B_1)\)  and \(\|\tilde{u}\|_{W^{2, 2}(B_1)} \leq C \|\tilde{f}\|_{L^2(B_1)}\).	
	
Because of the compact inclusion \(W^{1, 2}_0 \subset L^2\), we have an orthonormal basis of \(L^2\) by eigenfunctions; see the proof of Lemma \ref{spectlemma}. Therefore, we can write \( \tilde{f} = \sum_{\lambda} f_{\lambda} \) with \(f_{\lambda} \in W^{1, 2}_0(B_1)\) and \(-\Delta f_{\lambda} = \lambda f_{\lambda}\). The eigenvalues form a discrete sequence of non-negative numbers tending to infinity. Note that any weakly harmonic function in \(W^{1, 2}_0\) must vanish identically. In particular, the lowest eigenvalue must be strictly positive. Let \(u_{\lambda} = \lambda^{-1} f_\lambda\), we will show that the series \(\sum_{\lambda} u_{\lambda} \) converges in \(W^{2, 2}(B_1) \cap W^{1, 2}_0(B_1) \).

By standard regularity theory, any eigenfunction \(u_{\lambda}\) is smooth outside the conical set. Moreover, on \(B^{\mathrm{reg}}\) we have the identity
\begin{equation*}
\frac{1}{2} \Delta |\nabla u_{\lambda}|^2 = |\mbox{Hess}(u_{\lambda})|^2 - \lambda |\nabla u_{\lambda}|^2 .
\end{equation*}
Take a cut-off function \(\chi_{\epsilon}\), vanishing in a small neighborhood of the conical set and equal to \(1\) outside the \(\epsilon\)-tubular neighborhood. Multiply by \(\chi_{\epsilon}\) and integrate to get
\begin{equation*}
\int_{B_1} \chi_{\epsilon} |\mbox{Hess}(u_{\lambda})|^2 = \lambda \int_{B_1} \chi_{\epsilon} |\nabla u_{\lambda}|^2 + \frac{1}{2} \int_{B_1} (\Delta \chi_{\epsilon}) |\nabla u|^2 .
\end{equation*}
We let \(\epsilon \to 0\) and take \(\chi_{\epsilon}\) such that \(\lim_{\epsilon \to 0} \int_{B_1}|\Delta \chi_{\epsilon}| = 0\). Since \(|\nabla u_{\lambda} | \in L^{\infty}\), we conclude that \(u_{\lambda} \in W^{2, 2}(B_1) \) and
\begin{align*}
\int_{B_1} |\mbox{Hess}(u_{\lambda})|^2 &= \lambda \int_{B_1} |\nabla u_{\lambda}|^2 \\ &= \lambda^2 \int_{B_1} u_{\lambda}^2  \\ &= \int_{B_1} f_{\lambda}^2 .
\end{align*}

We can proceed the same way with finite linear combinations of eigenfunctions, to show that
\begin{equation*}
	\int_{B_1} |\mbox{Hess}(\sum u_{\lambda})|^2 = \sum \int_{B_1} f_{\lambda}^2
\end{equation*}
and similarly, \(\int_{B_1} (\sum u_{\lambda})^2 = \sum \lambda^{-2} \int_{B_1} f_{\lambda}^2 \) and \(\int_{B_1} |\nabla(\sum  u_{\lambda})|^2 = \sum \lambda^{-1} \int_{B_1} f_{\lambda}^2 \). Therefore, the series \(\sum_{\lambda} u_{\lambda}\) converges to a function \( u' \in W^{2, 2}(B_1)\cap W^{1, 2}_0(B_1)\) that satisfies \(-\Delta u' = f\) and \(\|u'\|_{W^{2, 2}(B_1)} \leq C \|\tilde{f}\|_{L^2(B_1)} \). By uniqueness, \(u' = \tilde{u}\) and we have shown the desired estimate.	
	
\end{proof}

\subsection{Scaled \(L^2\)-norms}

Given a geodesic ball \(B(x, \rho) \subset \R^m_{(\beta)} \), we set
\begin{equation*}
	 \|u\|_{B(x, \rho)} = \left(\rho^{-m} \int_{B(x, \rho)} u^2\right)^{1/2} .
\end{equation*}
This norm has the nice property of being scale invariant. More precisely, our space \(\R^m_{(\beta)}\) is a cone, as we will explain with more detail on Section \ref{prodconessect}, with dilations  given by
\begin{equation*}
	\lambda (r_a, \theta_a, s_i) = (\lambda r_a, \theta_a, \lambda s_i) .
\end{equation*}
Dilations scale the distance by \(d(\lambda x, \lambda y) = \lambda d(x, y)\), and \(\lambda B(x, \rho) = B(\lambda x, \lambda \rho) \). Given a function \(u: B(\lambda x, \lambda \rho) \to \R\), we can pull-it back via the dilation map \(\lambda: B(x, \rho) \to B(\lambda x, \lambda \rho) \), to \(u(\lambda \cdot): B(x, \rho) \to \R\). It is immediate to check that
\begin{equation*}
	\|u(\lambda \cdot)\|_{B(x, \rho)} = \|u\|_{B(\lambda x, \lambda \rho)} .
\end{equation*}

\subsubsection{Scaled estimates}
The Poincar\'e inequality, as well as the Caccioppoli inequality Lemma \ref{Caccioppoli} and the interior  \(L^2\)-bound on the Hessian in Proposition \ref{L2bound}, hold with a uniform constant for all balls \(B_1(p) \subset B_2(p)\) independent of the center. A scaling argument allows us to obtain the corresponding inequalities on balls of arbitrary radius. For example, the scale invariant version of estimate \eqref{hessbound} is
\begin{equation*}
	\|\mbox{Hess}(u)\|_{B(x, \rho)} \leq C \left(\| \Delta u\|_{B(x, 2\rho)} + \rho^{-2}\|u\|_{B(x, 2\rho)} \right)
\end{equation*}

\subsubsection{Campanato spaces}
The following characterization of the H\"older property, in terms of the decay of the \(L^2\)-norm, is attributed to Campanato; see \cite[Chapter 3.2]{HanLin}. We  denote averages by
\begin{equation*}
f_B = (\vol(B))^{-1} \int_B f .
\end{equation*}

\begin{lemma}
	Let \(\alpha>0\). Then there is a constant \(\kappa = \kappa(\alpha, \beta_a, m)\) with the following property: If \(f\) is in \(L^2(B_2)\) and there is  \(K>0\) such that
	\begin{equation*}
	\| f - f_{B_{\rho}(x)} \|_{B_{\rho}(x)} \leq K \rho^{\alpha}
	\end{equation*}
	for  all \(0<\rho<1\) and all \(x \in B_1\). Then \(f\) is continuous on \(B_1\) and 
	\begin{equation*}
	|f(x) - f(y)| \leq \kappa K d(x, y)^{\alpha} 
	\end{equation*}
	for every \(x, y \in B_1\) with \(d(x, y)<1/2\).
\end{lemma}

\begin{proof}
	
	Let \(x \in B_1\) and \(0<r<1/2\).
	
	\begin{align*}
	&|f_{B_{2r}(x)} - f_{B_r(x)} | \\ &= \frac{1}{\vol(B_{2r}(x))\vol(B_r(x))} \left| \int_{B_{2r}(x) \times B_r(x)} (f(u)-f(v)) dudv \right| \\
	&\leq \frac{1}{\vol(B_{2r}(x))^{\frac 1 2}\vol(B_r(x))^{\frac 12}} \left( \int_{B_{2r}(x) \times B_r(x)} (f(u)-f(v))^2 dudv \right)^{\frac 1 2} \\
	&\leq \frac{\sqrt{2}}{\vol(B_r(x))^{\frac 1 2}} \left( \int_{B_{2r}(x)} (f(u)-f_{B_{2r}(x)})^2 dudv \right)^{\frac 1 2} \\
	&\leq Q(\beta_a, m)\|f -f_{B_{2r}(x)} \|_{B_{2r}(x)} \\
	&\leq Q K (2r)^{\alpha} .
	\end{align*}
	
	Given \(0<\rho<1/2\), we apply the above to \(r = 2^{-k} \rho \) for \(k =0, 1, 2, \ldots, \infty \) to estimate \(\sum_{k=0}^{\infty} |f_{B_{2^{-k+1}\rho}} - f_{B_{2^{-k}\rho}}|\). We conclude that
	\begin{equation*}
	|f_{B_{\rho}(x)} - f(x)| \leq \kappa_1 K \rho^{\alpha} ,
	\end{equation*}
	with \(\kappa_1 = Q \left(\sum_{k=0}^{\infty}2^{-\alpha k}\right)\).
	
	Fix a pair of points \(x, y \in B_1\) with \(d = d(x, y)<1/2\). The same reasoning as before gives us
	\begin{align*}
	& |f_{B_d(x)} - f_{B_d(y)} | \\ &= \frac{1}{\vol(B_d(x))\vol(B_d(y))} \left| \int_{B_d(x) \times B_d(y)} (f(u)-f(v)) dudv \right| \\
	&\leq \frac{1}{\vol(B_d(x))^{1/2}\vol(B_d(y))^{1/2}} \left( \int_{B_{2d}(x) \times B_{2d}(x)} (f(u)-f(v))^2 dudv \right)^{1/2} \\
	&\leq \frac{\sqrt{2}\vol(B_{2d}(x))^{1/2}}{\vol(B_d(x))^{1/2}\vol(B_d(y))^{1/2}} \left( \int_{B_{2d}(x)} (f(u)-f_{B_{2d}(x)})^2 du \right)^{1/2} \\
	&\leq Q'(\beta_a, m) \|f -f_{B_{2d}(x)} \|_{B_{2d}(x)} \\
	&\leq Q' K (2d)^{\alpha} .
	\end{align*}
	We let \(\kappa_2 = Q' 2^{\alpha} \), so that \(|f_{B_d(x)} - f_{B_d(y)} | \leq \kappa_2 K d^{\alpha}\).
	Finally,
	\begin{align*}
	|f(x) - f(y)| &\leq |f(x) - f_{B_d(x)}| + |f_{B_d(x)} - f_{B_d(y)} | + |f_{B_d(y)} - f(y)| \\
	&\leq (2 \kappa_1 + \kappa_2) K d^{\alpha}
	\end{align*}
	which proves the claim.
\end{proof}

\begin{lemma} \label{auxlem}
	Let \(0 < \lambda < 1\) and let \(f \in L^2(B(x, 1))\). If for all integers \(k \geq 0\) there are \(\tau_k \in \R\) such that
	\begin{equation*}
		\|f-(\tau_1 + \ldots + \tau_k)\|_{B(x, \lambda^k)} \leq C \lambda^{k\alpha} ;
	\end{equation*}
	then \(\sum_k |\tau_k| < \infty\). Moreover, if we set \(\tau = \sum_k \tau_k\), then we have
	\begin{equation*}
		\|f - \tau\|_{B(x, r)} \leq C' r^{\alpha}
	\end{equation*}
	for all \(0< r < 1\) and some \(C' = \kappa(\alpha, \lambda)C\).
\end{lemma}

The proof is elementary and we omit it. Under the hypothesis of Lemma \ref{auxlem}, the Campanato condition holds. Because, for any \(\tau \in \R\), \(\|f - f_B\|_B \leq \| f - \tau\|_{B}\).

\section{Homogeneous harmonic functions on products of cones} \label{Hom}

\subsection{Products of cones} \label{prodconessect}

The point we want to emphasize is that our space
\[ \R^m_{(\beta)} = C(S^1_{2\pi\beta_1}) \times \ldots \times  C(S^1_{2\pi\beta_n}) \times \R^{m-2n}  \]
is itself a Riemannian cone. This follows from the general fact that the product \(C(L_1) \times C(L_2)\) of two cones is a cone \(C(L)\). Indeed, we write
\(g_{C_1} = dr_1^2 + r_1^2 g_{L_1} \) and \(g_{C_2} = dr_2^2 + r_2^2 g_{L_2} \), so that \((L_1, g_{L_1})\) and \((L_2, g_{L_2})\) are the respective links; and let \(g_C = g_{C_1} + g_{C_2}\). We introduce coordinates \((\rho, \psi) \in (0, +\infty) \times (0, \pi/2)\) by setting \(r_1 = \rho \sin \psi\) and \(r_2 = \rho \cos \psi\); it is then easy to check that
\begin{equation*}
	g_C = d\rho^2 + \rho^2 \left(d\psi^2 + (\sin^2 \psi) g_{L_1} + (\cos^2 \psi) g_{L_2}\right) .
\end{equation*}
We can rewrite this as
\(g_C = d\rho^2 + \rho^2 g_L\), where the link is \(L = [0, \pi/2] \times L_1 \times L_2 / (\{0\} \times L_1) \sqcup (\{\pi/2\} \times L_2) \) endowed with \(g_L = d\psi^2 + (\sin^2 \psi) g_{L_1} + (\cos^2 \psi) g_{L_2}\). At \(\psi =0\), the \(L_1\) factor in \(L\) collapses, so that \(L \cap \{\psi=0\}\) is a copy of \(L_2\) and the metric \(g_L\) has a singularity modeled on \(C(L_1)\). Similarly, at \(\psi=\pi/2\) the metric \(g_L\) has a singularity modeled on \(C(L_2)\) along \(L_1 = L \cap \{\psi=\pi/2\}\).

In our particular case, \(g\) is a cone over a singular metric \(\bar{g}\) on the \((m-1)\)-sphere. 
We write 
\begin{equation} \label{metricascone}
	g = d\rho^2 + \rho^2 \bar{g} ,
\end{equation}
with 
\[\rho^2 = \sum_a r_a^2 + \sum_i s_i^2 ;\] 
and we identify 
\begin{equation*}
	\R^m_{(\beta)} = C(S^{m-1}_{\bar{g}}) .
\end{equation*}
Our \(\bar{g}\) is a singular metric on the \((m-1)\)-sphere, with cone angles \(2\pi\beta_a\) along the copies of \(S^{m-3} \subset S^{m-1}\) cut-out by \(r_a=0\). On its regular part,  \(\bar{g}\) is locally isometric to the round sphere of curvature \(1\). As model examples, in dimensions \(m=3, 4\), we have:
\begin{itemize}
	\item \(C(S^1_{2\pi\beta}) \times \R\), so 
	\[\bar{g} = d\psi^2 + \beta^2 \sin^2\psi d\theta^2 \] 
	-with \(\psi \in (0, \pi)\)-
	is the `rugby ball' metric on the two-sphere with two antipodal points of angle \(2\pi\beta\).
	\item \(C(S^1_{2\pi\beta_1}) \times C(S^1_{2\pi\beta_2})\), so 
	\[\bar{g} = d\psi^2 + \beta_1^2 \sin^2\psi d\theta_1^2 + \beta_2^2 \cos^2\psi d\theta_2^2 \]
	-with \(\psi \in (0, \pi/2)\)- 
	is a constant curvature \(1\) metric on the three-sphere with cone angles \(2\pi\beta_1\) and \(2\pi\beta_2\) along two Hopf circles of lengths \(2\pi\beta_2\) and \(2\pi\beta_1\) respectively.
\end{itemize}

The dilations of \(g\) are given, in spherical coordinates \((\rho, \Theta) \in (0, + \infty) \times S^{m-1}\), by \(\lambda (\rho, \Theta) =(\lambda \rho, \Theta)\). In terms of the \((r_a, \theta_a, s_i)\) coordinates, we have 
\[\lambda (r_a, \theta_a, s_i) = (\lambda r_a, \theta_a, \lambda s_i) . \]

\subsection{Spectral theory} \label{spectthry}
Let \(\bar{g}\) be the singular metric on the sphere, as in equation \eqref{metricascone}, equivalently \(\bar{g}\) is the restriction of \(g\) to \(S^{m-1} = \{ \sum_a r_a^2 + \sum_i s_i^2 =1 \} \subset \R^m_{(\beta)}\). Clearly, \(\bar{g}\) is uniformly equivalent to the smooth round metric \(g_{\R^m}|_{S^{m-1}}\). Same as before, we have a standard Sobolev space \(W^{1, 2}(S^{m-1})\)
 
\begin{lemma}\label{spectlemma}
	There is an orthonormal basis \(\{\phi_i\}\) of \(L^2(S^{m-1})\) given by eigenfunctions of \(\Delta_{\bar{g}}\). More precisely, \(\phi_i \in W^{1, 2}(S^{m-1})\) solve \(-\Delta_{\bar{g}} \phi_i = \lambda_i \phi_i\) with \(0 = \lambda_0 < \lambda_1 \leq \lambda_2 \leq \ldots \) and \(\lambda_i \to \infty\).
\end{lemma}
\begin{proof}
	Recall that \(\|u\|_{W^{1, 2}} = \int u^2 + \int | \nabla_{\bar{g}} u |^2\).
	Given \(f \in L^2\), it defines a bounded linear functional on \(W^{1,2}\) by \(T(\phi) = \int f\phi\). If \(u\) is such that \( T = \langle u, - \rangle_{W^{1, 2}}\), then \(u\) is  a weak solution of \(-\triangle_{\bar{g}} u + u = f\). The map \(f \to u\) is a bounded linear map from \(L^2\) to \(W^{1, 2}\), composing this map with the compact inclusion we have a map \(K: L^2 \to L^2\) which is compact and self-adjoint. It follows from the spectral theorem that we can find an orthonormal basis \(\lbrace \phi_i \rbrace_{i\geq 0}\) of \(L^2\) such that \(K(\phi_i) = s_i \phi_i\) and \(s_i \to 0\). Unwinding the definitions, we get that \( \triangle_{\bar{g}} \phi_i = - \lambda_i \phi_i\) with \(0= \lambda_0 \leq \lambda_1 \leq \lambda_2 \leq \ldots\) and \(\lambda_i = (1-s_i)/s_i \to \infty\).
\end{proof}

The set of indicial roots is a discrete subset \(\mathcal{I} \subset \R\) given by the numbers \(d\), such that
\[\lambda = d (d + m-2) \]
is an eigenvalue of \(-\Delta_{\bar{g}}\) as in Lemma \ref{spectlemma}. In particular, since \(\lambda \geq 0\) we always have \(\mathcal{I}\cap (2-m, 0) = \emptyset\). For each eigenvalue \(\lambda\) we have two indicial roots associated to it, \(d_{\pm}\), with \(d_{+} \geq 0\) and \(d_{-} \leq 2-m\). If \(\phi_{\lambda}\) is a corresponding eigenfunction, then \(\rho^{d_{\pm}}\phi_{\lambda}\) are homogeneous harmonic on \(C(S^{m-1}_{\bar{g}})\). We will be interested in harmonic functions bounded in a neighborhood of the vertex, so we will mainly consider non-negative indicial roots.

Let \(B_1 \subset C(S^{m-1}_{\bar{g}})\) be the unit ball centered at the vertex of the cone. In polar coordinates \((r_a, \theta_a, s_i)\), \(B_1 = \{\sum_a r_a^2 + \sum_i s_i^2 < 1\}\). Let us denote by \(S= \cup_a \{r_a=0\}\) the singular set of \(g\).

\begin{lemma}[Maximum principle]
	Let \(u \in C^2(B_1 \setminus S)\cap C^0(\overline{B}_1)\) satisfy \(\Delta u =0\) on \(B_1 \setminus S\) and \(u=0\) on \(\p B_1\), then \(u=0\) on \(B_1\).	
\end{lemma}

\begin{proof}
	We reproduce the barrier function argument from \cite{GuoSongI}. For \(\epsilon>0\) we let 
	\[u_{\epsilon} = u + \epsilon \left(\sum_a \log r_a \right) . \]
	Clearly, \(\Delta u_{\epsilon} = 0\) on \(B_1 \setminus S\) and \(u_{\epsilon}(x) \to -\infty\) as \(x \to S\). By the standard maximum principle, we have \( \max_{\overline{B}_1} u_{\epsilon} = \max_{\p B_1} u_{\epsilon}\). On the other hand, \(u_{\epsilon} \leq u =0\) on \(\p B_1\). We conclude that for any \(x \in B_1 \setminus S\), we have \(u(x) = \lim_{\epsilon \to 0} u_{\epsilon}(x) \leq 0\). Similarly, \((-u)(x) \leq 0\), so \(u \equiv 0\).
\end{proof}

\begin{lemma} \label{expansionlemma}
	If \(u \in W^{1, 2}(B_1)\) is harmonic, we have
	\[u = \sum_{d_i \geq 0} \rho^{d_i} \phi_i . \]
\end{lemma}

\begin{proof}
	Expand \(u|_{S^{m-1}}\) into eigenfunctions, \[u|_{S^{m-1}} = \sum_{d_i \geq 0} \phi_i. \] Then \(u - \sum_{d_i \geq 0} \rho^{d_i} \phi_i \) is harmonic on \(B_1\) and vanishes identically on \(S^{m-1}\), by the maximum principle it must be identically zero.
\end{proof}

Let \(\mathcal{H}_d\) be the space of homogeneous harmonic functions of degree \(d\), so this is the zero vector space if \(d\) is not an indicial root. If \(d \in \mathcal{I}\), then \(\mathcal{H}_d\)  is identified with the corresponding space of eigenfunctions.
For \(d \geq 0\), we write \(\mathcal{H}_{\leq d}\) for the finite dimensional vector space spanned by homogeneous harmonic functions with degrees $\leq d$. Clearly, if \(d'>d\) then \(\mathcal{H}_d \subset \mathcal{H}_{d'}\). On the other hand, if  there are no indicial roots in the range $(d,d']$, then \(\mathcal{H}_{\leq d'} = \mathcal{H}_{\leq d}\). So, if we look at \(\mathcal{H}_{\leq d}\) as \(d \geq 0\) increases we see the following. We start with the constants \(\mathcal{H}_{\leq d} = \R\) for \(0 \leq d < d_1\), where \(d_1>0\) is the first positive indicial root. We have \(\mathcal{H}_{\leq d} = \R \oplus \mathcal{H}_{d_1} \) for \(d_1 \leq d < d_2\); and so on. By orthogonality of the eigenfunctions, the spaces \(\mathcal{H}_{d_i}\) and \(\mathcal{H}_{d_j}\), with \(i \neq j\), are orthogonal in \(L^2(B_1)\).

\subsection{Separation of variables} \label{spectrum1}

Consider the product of two cones, as in Section \ref{prodconessect}. Let \(h=g_L\) be the metric on its link, given by
\[h = d\psi^2 + (\sin^2 \psi) g_{L_1} + (\cos^2 \psi) g_{L_2} .\]
Recall that \(\psi \in (0, \pi/2)\) and the radial coordinates of the two factors are \(r_1=\rho \sin \psi\), \(r_2 = \rho \cos \psi\). We proceed to analyze the spectrum of \(\Delta_h\) by separation of variables, we follow \cite[Section 3.6]{ACM}. Write \(n_i = \dim L_i\) and let \(\phi_{\lambda_1}\), \(\phi_{\lambda_2}\) be eigenfunctions on the corresponding links, that is \(-\Delta_{g_{L_i}} \phi_i =  \lambda_i \phi_i\). We have 
\begin{equation*}
	-\Delta_h = \oplus_{\lambda_1, \lambda_2} L_{\lambda_1, \lambda_2} ,
\end{equation*}
where \(L_{\lambda_1, \lambda_2}\) acts on functions \(f(\psi) \phi_{\lambda_1} \phi_{\lambda_2}\) by
\begin{equation*}
	L_{\lambda_1, \lambda_2} = - \frac{\p^2}{\p \psi^2} - \left( n_1 (\tan \psi)^{-1} - n_2 \tan \psi \right) \frac{\p}{\p \psi} + \frac{\lambda_1}{\sin^2\psi} + \frac{\lambda_2}{\cos^2 \psi} .
\end{equation*}
We are led to analyze the eigenfunctions of \(L_{\lambda_1, \lambda_2}\), which can be done via the Sturm-Liouville theory for ODE's.

We consider the space of smooth compactly supported functions in \((0, \pi/2)\), endowed with the measure \(d\mu = \sin^{n_1}\psi \cos^{n_2}\psi d\psi\) and quadratic form
\begin{equation*}
	\langle f, f \rangle_{\lambda_1, \lambda_2} =
	\int_{0}^{\pi/2} \left( \left|\frac{\p f}{\p \psi} \right|^2 + f^2 \frac{\lambda_1}{\sin^2 \psi} + f^2 \frac{\lambda_2}{\cos^2\psi} \right) d\mu .
\end{equation*}
Then \(L_{\lambda_1, \lambda_2}\) is the associated self-adjoint operator with respect to the \(L^2\)-inner product, that is
\begin{equation*}
	\langle f, g \rangle_{\lambda_1, \lambda_2} = \int_{0}^{\pi/2} (L_{\lambda_1, \lambda_2} f) g d\mu .
\end{equation*}

We want \(f(\psi) \phi_{\lambda_1}\phi_{\lambda_2} \in W^{1, 2}((L, h))\). Since smooth functions with compact support on the regular part are dense in \(W^{1, 2}\), this means that we want \(f\) to be in the closure of compactly supported functions on \((0, \pi/2)\) with respect to  \(\langle \cdot, \cdot \rangle_{\lambda_1, \lambda_2} + \|\cdot\|^2_{L^2(d\mu)}\). By standard theory, \(L_{\lambda_1, \lambda_2}\) has discrete spectrum; actually we already proved that \(\Delta_{\bar{g}}\) has discrete spectrum. Each eigenvalue has exactly one (up to scalar factor) eigenfunction \(f\), smooth on \((0, \pi/2)\), and which satisfies  \(f(\psi) \phi_{\lambda_1}\phi_{\lambda_2} \in W^{1, 2}((L, h))\). Indeed, changing variables to \(x=\sin^2\psi\), the ODE has regular singularities at the end-points \(x=0\) and \(x=1\); see equation \eqref{ODE}. The eigenfunctions of \(L_{0,0}\) are hypergeometric functions and our claim on multiplicity one eigenvalues is easily verified. On the other hand, if \(\lambda_1>0\), then the indicial root equation of \(L_{\lambda_1, \lambda_2} f = \lambda f\) at \(x=0\) is 
\begin{equation*}
    r \left( r-1 + \frac{1+n_1}{2} \right) = \frac{\lambda_1}{4} .
\end{equation*}
This has two solutions \(r_1>0\) and \(r_2<0\) which are the leading order terms in the Frobenius series solutions, only the one with \(r_1>0\) extends continuously over \(x=0\). The same discussion applies for \(L_{0, \lambda_2}\) at \(x=1\) and \(\lambda_2>0\).
The first eigenvalue of \(L_{\lambda_1, \lambda_2}\) is clearly \(\geq \lambda_1 + \lambda_2\) and its associated eigenfunction is nowhere vanishing in \((0, \pi/2)\). The second eigenfunction vanishes exactly once on \((0, \pi/2)\), and so on.  Note that, since \(f(\psi)\phi_{\lambda_1}\phi_{\lambda_2}\) extends continuously to \(L\), we have that \(f(0)=0\) if \(\lambda_1 \neq 0\) and likewise \(f(\pi/2)=0\) if \(\lambda_2 \neq 0\); which matches with the Frobenius solution having a positive power leading term.

\subsection{The operators \(L_{\lambda_1, \lambda_2}\)} \label{spectrum2}

Let \(x = \sin^2\psi\), so \(x\) varies in the interval \((0,1)\) and 
\begin{equation} \label{ODE}
	L_{\lambda_1, \lambda_2} = -4x(1-x) \frac{\p^2}{\p x^2} + 2\left((2+n_1+n_2)x -1 -n_1 \right) \frac{\p}{\p x} + \frac{\lambda_1}{x} + \frac{\lambda_2}{1-x}
\end{equation}

We look first at \(L_{0,0}\). Its eigenfunctions correspond to homogeneous harmonic functions which only depend on the radial coordinates of the two factors. That is, \(u = u(r_1, r_2)\) and
\begin{equation*}
	\left(\frac{\p^2}{\p r_1^2} + \frac{n_1}{r_1}\frac{\p}{\p r_1} + \frac{\p^2}{\p r_2^2} + \frac{n_2}{r_2}\frac{\p}{\p r_2} \right) u(r_1, r_2) = 0 .
\end{equation*}
The first eigenvalue of \(L_{0,0}\) is \(0\), with constant eigenfunctions. To find the second eigenvalue we note the following:
\begin{equation*}
	\Delta_{C(L)} r_a^2 = 2 (n_a+1), \hspace{2mm} a= 1, 2 .
\end{equation*}
Therefore, the function \(u = (n_1 + 1)^{-1} r_1^2 - (n_2+1)^{-1} r_2^2\) is a degree $2$ homogeneous harmonic function on \(C(L)\). We write \(u = \rho^2 f(\psi)\) with
\begin{align*}
	f(\psi) &= (n_1+1)^{-1} \sin^2 \psi - (n_2 + 1)^{-1} \cos^2 \psi \\
	&= \left(\frac{1}{n_1+1} + \frac{1}{n_2 +1}\right)x - \frac{1}{n_2+1} . 
\end{align*} 
It is easy to check that \(L_{0, 0} f = -\lambda f\) with \(\lambda = 2(n_1 + n_2 + 2) \). This matches with the formula \(\lambda = d (d-2 + m)\) with \(m=\dim C(L) = n_1 + n_2 + 2\) and \(d=2\). Moreover, \(f(x)\) vanishes exactly once for \(x \in (0, 1)\), which means that \(\lambda =2m\) is the next eigenvalue after \(\lambda=0\). To find the next, we note that
\begin{equation*}
	 \Delta_{C(L)} r_a^4 = 4(3+n_a)r_a^2, \hspace{2mm} \Delta_{C(L)} (r_1^2r_2^2) = 2(n_1+1)r_2^2 + 2(n_2+1)r_1^2 .
\end{equation*}
Therefore,
\begin{align*}
	u &= \frac{1+n_2}{3+n_1}r_1^4 + \frac{1+ n_1}{3+n_2} r_2^4 - 2r_1^2r_2^2 \\
	&= \rho^4 \left(\frac{1+n_2}{3+n_1}\sin^4\psi + \frac{1+ n_1}{3+n_2} \cos^4\psi - 2\sin^2\psi \cos^2\psi\right)
\end{align*}
is homogeneous harmonic of degree \(4\). This implies that
\begin{equation*}
	f = \frac{1+n_2}{3+n_1}x^2 + \frac{1+ n_1}{3+n_2} (1-x)^2 - 2x(1-x)
\end{equation*}
satisfies \(L_{0,0} f = \lambda f\), with \(\lambda = 4(4+n_1+n_2)\). Finally, we note that \(f(x)\) vanishes exactly twice for \(x \in (0, 1)\). To sum up, the first three eigenvalues of \(L_{0, 0}\) give rise to homogeneous harmonic functions of degrees \(0, 2\) and \(4\).

We look now at \(L_{\lambda_1, 0}\) for a non-zero eigenvalue of \(\Delta_{L_1}\). Let \(\phi_1\) be an eigenfunction, \(-\Delta_{L_1} \phi_1 = \lambda_1 \phi_1 \). Let \(\gamma_1 >0\) be given by \(\gamma_1 (\gamma_1 + n_1 -1) = \lambda_1\), so 
\begin{align*}
	u &= r_1^{\gamma_1} \phi_1 \\
	&= \rho^{\gamma_1} (\sin \psi)^{\gamma_1} \phi_{\lambda_1} 
\end{align*}
is a homogeneous harmonic function on \(C(L)\). Therefore, \(f(x) = x^{\gamma_1/2}\) satisfies \(L_{\lambda_1, 0} f = \lambda f\) with \(\lambda = \gamma_1 (\gamma_1 + m -2)\). Since \(f(x)\) in non-vanishing on \((0,1)\), it is the first eigenfunction. To find the next, we note the following:
\begin{equation*}
	\Delta_{C(L)} (r_1^{\gamma_1}r_2^2 \phi_{\lambda_1}) = 2(n_2+1)r_1^{\gamma_1} \phi_{\lambda_1}, \hspace{2mm} \Delta_{C(L)} (r_1^{\gamma_1 +2} \phi_{\lambda_1}) = 2(2\gamma_1+1+n_1)r_1^{\gamma_1} \phi_{\lambda_1} .
\end{equation*}
Therefore,
\begin{align*}
	u &= \left((n_2+1)^{-1}r_1^{\gamma_1}r_2^2 - (2\gamma_1+1+n_1)^{-1}r_1^{\gamma_1+2} \right) \phi_{\lambda_1} \\ 
	&= \rho^{\gamma_1+2} \left((n_2+1)^{-1}\sin^{\gamma_1}\psi \cos^2\psi - (2\gamma_1+1+n_1)^{-1}\sin^{\gamma_1+2}\psi \right) \phi_{\lambda_1}
\end{align*}
is homogeneous harmonic on \(C(L)\). We conclude that
\begin{equation*}
	f(x) = x^{\gamma_1/2} \left(\frac{1}{n_2+1} - \left(\frac{1}{n_2+1} + \frac{1}{2\gamma_1 + 1 + n_1}\right) x \right)
\end{equation*}
satisfies \(L_{\lambda_1, 0} f = \lambda f\) with \(\lambda = (\gamma_1+2)(\gamma_1 +m)\). Since \(f(x)\) vanishes only once in \((0,1)\), it is the second eigenfunction. To sum up, the first two eigenvalues of \(L_{\lambda_1,0}\) correspond to homogeneous harmonic functions on \(C(L)\) of degrees \(\gamma_1\) and \(\gamma_1 +2\). Clearly, the same discussion applies to \(L_{0, \lambda_2}\).

Finally, we consider \(L_{\lambda_1, \lambda_2}\) where \(\lambda_a\) are non-zero eigenvalues. Let \(\phi_a\) be corresponding eigenfunctions, so \(-\Delta_{L_a} = \lambda_a \phi_a\). We get that
\begin{align*}
	u &= r_1^{\gamma_1}r_2^{\gamma_2} \phi_{\lambda_1} \phi_{\lambda_2} \\
	&= \rho^{\gamma_1 + \gamma_2} \left(\sin^{\gamma_1}\psi \cos^{\gamma_2}\psi\right) \phi_{\lambda_1} \phi_{\lambda_2}
\end{align*}
is homogeneous harmonic on \(C(L)\). Hence
\begin{equation*}
	f(x) = x^{\gamma_1/2}(1-x)^{\gamma_2/2}
\end{equation*}
satisfies \(L_{\lambda_1, \lambda_2}f = \lambda f\) with \(\lambda = (\gamma_1+\gamma_2)(\gamma_1+\gamma_2-2+m)\). Since \(f\) is non-vanishing on \((0,1)\), this is the first eigenfunction. We conclude that all homogeneous harmonic functions associated to eigenvalues of \(L_{\lambda_1, \lambda_2}\) have degree at least \(\gamma_1 +\gamma_2\). 

\subsection{Sub-quadratic harmonic functions}

We identify the space of subquadratic harmonic `polynomials' \(\mathcal{H}_{\leq 2}\), as defined in Section \ref{spectthry}, for our space \(\R^m_{(\beta)} = C(S^{m-1}_{\bar{g}})\).

\begin{proposition} \label{subquadraticprop}
	\(\mathcal{H}_{\leq 2}\) is spanned by
	\begin{enumerate}
		\item \(1\), \(s_i\);
		\item \(r_a^{1/\beta_a} \cos (\theta_a) \), \(r_a^{1/\beta_a} \sin (\theta_a) \) for those \(\beta_a \geq 1/2\);
		\item \(\mathcal{H}_2(\R^{m-2n})\), \(r_a^2 - r_b^2\)  and \(r_a^2 - 2s_i^2\).
	\end{enumerate}
Moreover, \(\mathcal{H}_{\leq 2 + \alpha} = \mathcal{H}_{\leq 2}\) provided that \(0<\alpha<\mu\).
\end{proposition}

\begin{proof}
	
We assume \(m\geq 4\), for the special case \(m=3\) see Section \ref{specialcase} below. We proceed by induction on the number of cone factors, so consider first \(C(S^1_{2\pi\beta}) \times \R^{m-2}\). We apply the analysis in Sections \ref{spectrum1} and \ref{spectrum2} to \(L_1 = S^1_{2\pi\beta}\) and \(L_2=S^{m-3}(1)\). The operator \(L_{0, 0}\) gives rise to constants, \((1/2)r^2-(m-3)^{-1}|s|^2\) and then homogeneous harmonic functions of degree \(4\) are higher.

We proceed to analyze \(L_{\lambda_1, 0}\). The first non-zero eigenvalue of \(S^1_{2\pi\beta}\) is \(\lambda_1 = 1/\beta^2\) with eigenfunctions \(\cos \theta\) and \(\sin \theta\). The corresponding homogeneous harmonic functions \(r^{1/\beta}\cos\theta\), \(r^{1/\beta}\sin\theta\) have degree \(1/\beta\). If \(1/2 \leq \beta<1\), then these functions are subquadratic, otherwise, for $0 < \beta < 1/2$, they have degree $>2$. The other eigenvalues of \(L_{\lambda_1, 0}\) give rise to homogeneous harmonic functions of degree \(2+(1/\beta)\) and higher. The other eigenvalues on \(S^1_{2\pi\beta}\) are \(k^2/\beta^2\) with \(k\geq 2\) and give rise to homogeneous harmonic functions of degrees \(2/\beta> 2\) and higher. We consider \(L_{0, \lambda_2}\). The first non-zero eigenvalue of \(S^{m-3}(1)\) is \(\lambda_2 = m-3\). The corresponding homogeneous harmonic functions \(s_i\),  have degree \(1\). The other eigenvalues of \(L_{0,\lambda_2}\) give rise to homogeneous harmonic functions of degree \(3\) and higher. The next eigenvalue on \(S^{m-3}(1)\) is \(2(m-2)\), the corresponding homogeneous harmonic functions are \(\mathcal{H}_2(\R^{m-2})\). The other homogeneous harmonic functions coming from \(L_{0, 2(m-2)}\) have degree \(4\) and higher. Finally, the lowest eigenvalue of \(L_{\lambda_1, \lambda_2}\) gives rise to the homogeneous harmonic functions \(s_i r^{1/\beta} cos \theta\) and \(s_i r^{1/\beta} cos \theta\), which have degree \(1+1/\beta\). As a result, we see that the next indicial root after \(2\) is the minimum of \(3\), \(1/\beta\) (if \(\beta<1/2\)), \(1+1/\beta\) (if \(\beta>1/2\)). In particular, by our choice of \(\alpha\), there are no indicial roots in \((2, 2+\alpha]\) and the statement \(\mathcal{H}_{\leq 2} = \mathcal{H}_{\leq 2+ \alpha}\) follows.

More generally, we repeat the above argument with \(\R^{m-2}\) replaced by \(C(L_2)\) with \(\mathcal{H}_{\leq 2}(C(L_2))\) given as in the statement of the proposition. The first non-zero eigenvalue of \(L_2\) gives rise, through \(L_{0, \lambda_2}\), to homogeneous harmonic functions of degree \(1\), \(3\), and higher if there is at least one Euclidean factor on \(C(L_2)\); or degree \(1/\beta_a\), \(1/\beta_a + 2\), and higher if \(C(L_2)\) does not split off any Euclidean factor. The second non-zero eigenvalue of \(L_2\) gives rise to the degree \(2\) functions \(\mathcal{H}_2(\R^{m-2n})\), degree \(4\) and higher. If \(m=2n\), then the second non-zero eigenvalue of \(L_2\) gives rise to homogeneous harmonic functions of degree \(2/\beta_a>2\) and higher. Similarly, the homogeneous harmonic functions coming from eigenvalues of \(L_{\lambda_1, \lambda_2}\) have degree at least \(1/\beta_1 + 1\) if \(m>2n\), and \(1/\beta_1 + 1/\beta_a\) for some \(a\geq 2\) if \(m=2n\). As before, we see that the next indicial root after \(2\) is the minimum of \(3\), \(1/\beta_a\) for those \(\beta_a<1/2\), \(1+1/\beta_a\) for those \(\beta_a>1/2\). In particular, by our choice of \(\alpha\), there are no indicial roots in \((2, 2+\alpha]\) and the statement of Proposition \ref{subquadraticprop} follows.

\end{proof}

\begin{corollary} \label{subquadcor}
	For every \(D \in \mathcal{D}_2\) and every \(P \in \mathcal{H}_{\leq 2}\) we have \(DP \equiv \tau \in \R\). Moreover,
	\begin{equation} \label{polyest}
		|\tau| \leq C \|P\|_{L^2(B_1)}
	\end{equation}
	for some \(C=C(m,\beta_a)\).
\end{corollary}

\begin{proof}
	Given the explicit description of \(\mathcal{H}_{\leq 2}\) in Proposition \ref{subquadraticprop}, the fact that \(DP\) is constant follows by immediate check inspection. The estimate \eqref{polyest} follows from the general principle that any two norms on a finite dimensional vector space are equivalent.
\end{proof}

\subsubsection{Special case: \(m=3\)} \label{specialcase}

We present an alternative proof to Proposition \ref{subquadraticprop} for the case of \(C(S^1_{2\pi\beta}) \times \R\).  In cylindrical coordinates \((r, \theta, s)\),  the metric is \(g = dr^2 + \beta^2 r^2 d\theta^2 + ds^2\). In spherical coordinates, we introduce the angular variable \(t \in (0, \pi)\) and write \(r = \rho \sin t\), \(s = \rho \cos t\).  We have
\begin{align*}
	g &= d\rho^2 + \rho^2 (dt^2 + \beta^2 \sin^2 t d\theta^2) \\
	&= d\rho^2 + \rho^2 \bar{g} .
\end{align*}
Here, \(\bar{g} = dt^2 + \beta^2 \sin^2t d\theta^2\) is the spherical `rugby ball' metric on \(S^2\). Let \(d \geq 0\) and let \(u = \rho^d \varphi\) be a homogeneous harmonic function, so \(\varphi = \varphi(t, \theta) \in W^{1, 2}(S^2)\) solves \(-\Delta_{\bar{g}} \varphi = \lambda \varphi\). 

The function \(u\) is smooth away from \(r=0\) and \(\Delta (\p u/ \p s)=0\). On the other hand,
\[ \frac{\p}{\p s} = \cos t \frac{\p}{\p \rho} - \frac{\sin t}{\rho} \frac{\p}{\p t} .\]
Therefore, \(\p u / \p s = \rho^{d-1} \psi\) where \(\psi\) is a smooth function on \(S^2\) away from the two conical points and satisfies \(-\Delta_{\bar{g}} \psi = \tilde{\lambda} \psi\) point-wise on the regular part. It is given by
\begin{equation*}
	\psi = - (\sin t) \frac{\p \varphi}{\p t} + d (\cos t) \varphi . 
\end{equation*}

We claim that \(\psi\) extends as a H\"older continuous function over \(t=0\). Indeed, by De Giorgi-Nash-Moser we know that \(\varphi\) is \(C^{\alpha}\) at \(t=0\). Applying the standard Schauder estimates in small balls away from the conical point we get that \(|\p \varphi / \p t| = O (t^{\alpha-1})\). Therefore, \(\sin t (\p \varphi / \p t)\) extends by \(0\) over \(t=0\) as a \(C^{\alpha}\) function. Since \(\psi\) is \(C^{\alpha}\) and satisfies \(-\Delta_{\bar{g}} \psi = \lambda \psi\) away from the conical point, the standard Schauder theory implies that \(|\nabla_{\bar{g}} \psi| = O(t^{\alpha-1})\), so \(\psi \in W^{1, 2}\) and it solves the eigenvalue equation in the weak sense.

We can repeat the argument for higher order tangential derivatives. The conclusion is that, if \(\p^k u / \p s^k \neq 0\) then \(d-k\) is an indicial root. If \(d>0\) is not an integer we let \(k\) be the smallest integer such that \(d-k < 0\), so \(-1 < d-k <0\). Since there are no indicial roots in the interval \((-1, 0) = (2-m, 0)\), we conclude that \(\p^k u /\p s^k \equiv 0\) and \(u = u_0(r, \theta) p(s) \) for some polynomial \(p(s)\) of degree \(< d\). We see that, if \(d<2\) then \(p(s)\) must be affine and \(\Delta_{\beta} u_0 =0\). The possibilities are \(u =\) constant, or, if \(\beta>1/2\), a linear combination of \(r^{1/\beta}\cos\theta\), \(r^{1/\beta}\sin\theta\).

If \(d=2\), then \(\p^2 u / \p s^2\) is homogeneous harmonic of degree zero, hence a constant. After subtracting a multiple of \(r^2 - 2s^2\), we can assume   \(\p^2 u / \p s^2 =0\), so \(p(s)\) must be affine and \(\Delta_{\beta}u_0 =0\). The cone \(C(S^1_{2\pi\beta})\) does not have non-zero homogeneous harmonic functions of either degree \(1\) (because \(\beta<1\)) or \(2\) (because \(\beta \neq 1/2\)).
The statement of Proposition \ref{subquadraticprop} for the case \(C(S^1_{2\pi\beta}) \times \R\) then follows.

\section{Elementary quantitative considerations} \label{Quant}

\subsection{Geodesics balls} \label{saclesect}
Let \(\mathcal{C}\) be the finite collection of at most \(2^n\) cones \(C(Y)\), given by \(\R^m\), \(C(S^1_{2\pi\beta_i}) \times \R^{m-2}\) and more generally
\[C(S^1_{2\pi\beta_{i_1}}) \times \ldots \times C(S^1_{2\pi\beta_{i_{\ell}}}) \times \R^{m-2\ell} \]
with \( \ell \leq n\) and \(1 \leq i_1 < \ldots < i_{\ell} \leq n\). Here, \(Y\) denotes the \((m-1)\)-sphere endowed with a possibly singular metric as the ones described in Section \ref{prodconessect}. Equivalently, \(\mathcal{C}\) consists of all \(\R^m_{(\beta')}\) where \((\beta')\) is a subset of \((\beta)= (\beta_1, \ldots, \beta_n) \).

Given \(\epsilon>0\), we say that \(B(x, \rho) \subset \R^m_{(\beta)}\) is \(\epsilon\)-close to a ball \(B_{C(Y)}(0, \rho) \subset C(Y)\) for a cone \(C(Y) \in \mathcal{C}\), if after introducing cuts in some cone factors (see Appendix \ref{2cone}) \(B(x, \rho)\) is isometric to a ball \(B_{C(Y)}(z, \rho) \subset C(Y)\) with \(|z| = d_{C(Y)}(z, o) < \epsilon \rho\). Equivalently, \(\rho^{-1}B(x, \rho)\) is isometric to \(B_{C(Y)}(z, 1)\) with \(|z| < \epsilon\).

For $x \in \R^m_{(\beta)}$, consider the sequence of geodesic balls $B(x,\lambda^k)$ for $k \in \mathbb Z$ and constant $0<\lambda<1$. If $B(x,\lambda^k)$ is $\epsilon$-close to $B_{C(Y)}(0,\lambda^k)$ for some $C(Y) \in \mathcal C$ (equivalently, $\lambda^{-k} B(x,\lambda^k)$ is $\epsilon$-close to $B_{C(Y)}(0,1)$), we say $\lambda^k$ is a \textit{good scale} for $x$. If not, then we call $\lambda^k$ a \textit{bad scale}.

If $\lambda^k$ is a good scale for $x$, and so the rescaled ball, $\lambda^{-k} B(x,\lambda^k) = B(x_k,1)$ for $x_k = \lambda^{-k} x$, is isometric to $B_{C(Y)}(z,1) \subset C(Y)$ with $|z|<\epsilon$, we sometimes identify $B(x_k,1)$ with its image $B_{C(Y)}(z,1)$ and write $B(x_k,1) \subset C(Y)$ with $|x_k|<\epsilon$.

\begin{lemma}[Finitely many bad scales]\label{scaleslemma}
	Fix \(0 < \lambda < 1\) and \(\epsilon_0>0\). Then there is \(N = N(\epsilon_0, \lambda)\) such that for each \(x \in \R^m_{(\beta)}\), and every integer $k$ except for at most \(N\) bad scales, the geodesic ball \(B(x, \lambda^k)\) is \(\epsilon_0\)-close to a ball \(B_{C(Y)}(0, \lambda^k) \subset C(Y)\) for some \(C(Y) \in \mathcal{C}\). 
\end{lemma}

\begin{proof}
	We fix some \(c = c(\beta) > 1\) such that the following holds. Any unit geodesic ball in \(C(S^1_{2\pi\beta})\) whose center is at distance \(> c\) from the vertex, is isometric -after a cut- to the unit Euclidean ball in \(\R^2\).
	
	Without loss of generality, we can assume that \(r_1(x) \leq r_2(x) \leq \ldots \leq r_n(x)\) and \(s=0\);  we will write \(r_a = r_a(x)\). We rescale \(B(x, \lambda^k)\) to unit size, so we end up with \(B(x_k , 1)\) where \(x_k = \lambda^{-k} x\). The coordinates of \(x_k\) are \(r_a(x_k) = \lambda^{-k} r_a\).
	If \(r_1(x_k) = \lambda^{-k} r_1 > c\), that is
	\[ (\log\lambda^{-1}) k > \log r_1^{-1} + \log c , \]
	then \(B(x, 1)\) is \(\epsilon_0\)-close (indeed isometric) to the unit ball in \(\R^m\). If  \(\lambda^{-k}r_2 > c\)  and \(\lambda^{-k}r_1 < \epsilon_0\), that is
	\[ \log c + \log r_2^{-1} < (\log \lambda^{-1}) k < \log r_1^{-1} - \log \epsilon_0^{-1} , \]
	then \(B(x_k, 1)\) is \(\epsilon_0\)-close to the unit ball in \(C(S^1_{2\pi\beta_1}) \times \R^{m-2}\). 
	If \(\lambda^{-k}r_3 > c\)  and \(\lambda^{-k}r_2 < \epsilon_0\), that is
	\[ \log c + \log r_3^{-1} < (\log \lambda^{-1}) k < \log r_2^{-1} - \log \epsilon_0^{-1} , \]
	then \(B(x_k, 1)\) is \(\sqrt{2}\epsilon_0\)-close to the unit ball in \(C(S^1_{2\pi\beta_1}) \times C(S^1_{2\pi\beta_2}) \times \R^{m-4}\).
	We proceed likewise; if \(\lambda^{-k} r_n < \epsilon_0\), so 
	\[ (\log\lambda^{-1}) k < \log r_n^{-1} - \log \epsilon_0^{-1} , \]
	then \(B(x_k, 1)\) is \(\sqrt{d}\epsilon_0\)-close to the unit ball in \(C(S^1_{2\pi\beta_1}) \times \ldots \times  C(S^1_{2\pi\beta_n}) \times \R^{m-2n}\).  
	
	We conclude that the set of \(k \in \mathbf{Z}\) for which \(B(x, \lambda^k)\) is not \(\sqrt{n}\epsilon_0\)-close to a ball in some cone \(C(Y) \in \mathcal{C}\), is contained in a union of at most \(n\) intervals of length \((\log \lambda^{-1})^{-1} \left(\log \epsilon_0^{-1} + \log c \right) \). In particular, the lemma follows with
	\[ N = n (\log \lambda^{-1})^{-1} \left(\log \epsilon_0^{-1} + \log c + (1/2) \log n \right) . \]

\end{proof}

\subsection{Monotonicity of scaled \(L^2\)-norms} 

Recall that \(B_{\rho} \subset C(S^{m-1}_{\bar{g}})\) is the ball centered at zero of radius \(\rho\).

\begin{lemma}
	Let \(d \geq 0\) and assume that \(u \in W^{1, 2}(B_1)\)  is harmonic and \(L^2\)-orthogonal to \(\mathcal{H}_{\leq d}\), which means \(u = \sum_{d_i > d} \rho^{d_i} \phi_i\). Let \(d_{*} = \min \{d_i > d\} \). Then, for any \(0<\rho<1\), we have
	\[\|u\|_{B_{\rho}} \leq \rho^{d_*} \|u\|_{B_1} , \]
	with equality if and only if \(u\) is homogeneous of degree \(d_*\).
\end{lemma}

\begin{proof}
	By the orthogonality of eigenfunctions, we have
	\begin{align*}
	\int_{B_{\rho}} u^2 &= \sum_{i, j} \int_{B_\rho} r^{d_i} r^{d_j} \phi_i \phi_j \\
	&= \sum_i \frac{1}{2d_i + m}\rho^{2d_i + m} \int_{S^{m-1}} \phi_i^2 .
	\end{align*}
	From here,
	\begin{align*}
	\|u\|_{B_{\rho}}^2 &= \rho^{2d_*} \sum_i \frac{1}{2d_i +m} \rho^{2d_i-2d_*} \int_{S^{m-1}} \phi_i^2 \\
	&\leq \rho^{2d_*} \|u\|_{B_1}^2 .
	\end{align*}
	with equality if and only if $u$ is homogeneous of degree $d_*$.
\end{proof}

Next, we provide a version of the above lemma, considering balls which are not necessarily centered at the vertex. In particular, this establishes the quantitative decay property for harmonic functions orthogonal to \(\mathcal{H}_{\leq 2}\) for any good scale $B(x,\lambda^k)$.

Recall that, for \(x \in C(Y)\), we use \(|x|\) to denote the distance to the vertex and \(\mathcal{C}\) denotes the finite collection of model cones defined in Section \ref{saclesect}.

\begin{lemma} \label{monlemma}
	Let \(0 < \alpha < \mu\) and let \(0 < \lambda < 1\). There is \(\epsilon_0 = \epsilon_0 (\alpha, \lambda)\) with the following property. If \(x \in C(Y)\) for some \(C(Y) \in \mathcal{C}\) and \(|x| < \epsilon_0\), then for every harmonic function \(u\) in \(B(x, 1)\) which is \(L^2\)-orthogonal to \(\mathcal{H}_{\leq 2}(C(Y))\) we have
	\begin{equation*}
		\|u\|_{B(x, \lambda)} < \lambda^{2+ \alpha} \|u\|_{B(x, 1)} .
	\end{equation*}
\end{lemma}

\begin{proof}
	Let \(x \in C(Y)\) have \(|x|< \epsilon\). Since \(B(x, \lambda) \subset B(0, \lambda + \epsilon)\), we clearly have
	\begin{equation} \label{eqmon1}
		\|u\|^2_{B(x, \lambda)} \leq \left(1 + \epsilon/\lambda\right)^m \|u\|^2_{B(0, \lambda + \epsilon)} .
	\end{equation}
	Since \(u\) is \(L^2\)-orthogonal to \(\mathcal{H}_{\leq 2}\), we have
	\begin{equation} \label{eqmon2}
		\|u\|^2_{B(0, \rho_1)} \leq (\rho_1/\rho_2)^{2d_*} \| u\|^2_{B(0, \rho_2)} ,
	\end{equation}
	where \(0 < \rho_1 < \rho_2\) and \(d_*\) is the smallest indicial root bigger than \(2\). We take \(\rho_1 = \lambda + \epsilon\) and \(\rho_2 = 1 - \epsilon\). Equations \eqref{eqmon1} and \eqref{eqmon2}, together with \(\|u\|_{B(0, 1-\epsilon)} \leq (1-\epsilon)^{-m} \|u\|_{B(1, x)} \), give us
	\begin{equation*}
		\|u\|_{B(x, \lambda)} \leq \lambda^{2+\alpha} F(\epsilon) \|u\|_{B(x, 1)} 
	\end{equation*}
	with 
	\[F(\epsilon) = \lambda^{d_* -2 -\alpha} \left( 1 + \epsilon/\lambda\right)^{m+d_*} \left(1-\epsilon\right)^{-m-d_*} .\]
	Note that \(2+\alpha < d_{*}\), by our choice of \(\alpha\).
	Since \(F\) is continuous with respect to \(\epsilon < 1\) and \(F(0) = \lambda^{d_* -2 -\alpha} < 1\), it suffices to take \(\epsilon_0\) such that \(F(\epsilon) < 1\) for all \(\epsilon < \epsilon_0\).
\end{proof}

\section{Proof of Theorem \ref{mainthm}}

As before, \(\mathcal{C}\) denotes the finite collection of model cones defined in Section \ref{saclesect} and, for \(x \in C(Y)\), we use \(|x|\) to denote the distance to the vertex. 

\begin{proposition}[Key step]\label{keystep}
	Let \(0< \alpha <\mu\). There are \(0 < \lambda < 1\), \(\delta>0\) and \(\epsilon_0>0\) with the following property. If \(x \in C(Y)\) for some \(C(Y) \in \mathcal{C}\), \(|x| < \epsilon_0\) and \(\Delta u = f \) with \(\|f\|_{B_1(x)} < \delta \), \(\|u\|_{B_1(x)} \leq 1\). Then there is \(P \in \mathcal{H}_{\leq 2} (C(Y)) \) such that
	\[\|u - P\|_{B_{\lambda}(x)} < \lambda^{2+\alpha} .\]
	Moreover, \(\|P\|_{B_1} \leq C\).
\end{proposition}

\begin{proof}
	We take \(\|f\|_{B_1(x)} < \delta\) small so that, by the harmonic approximation Lemma \ref{harmaproxlem}, there is \(h \in W^{1, 2}(B(x, 1/2))\) harmonic with \(\|u-h\|_{B(x, 1/2)} < \epsilon\) for some \(\epsilon\) to be determined later. Decompose \(h\) into its sub-quadratic and super-quadratic parts, that is \( h = h_{\leq 2} + h_{>2} \) with \(h_{\leq 2} \in \mathcal{H}_{\leq 2}\) and \(h_{>2}\) being \(L^2\)-orthogonal to \(\mathcal{H}_{\leq 2}\). Clearly, \(\|h\|_{B(x, 1/2)} \leq \|u-h\|_{B(x,1/2)} + \|u\|_{B(x,1/2)} \leq C\) and the same holds for its components \(h_{\leq 2}\), \(h_{>2}\).
	
	Fix \(\alpha < \alpha' < \mu\) and let \(0 < \lambda < 1/2\) to be determined. Let \(\epsilon_0 = \epsilon_0' (\alpha', \lambda')/2\) with \(\epsilon_0'(\alpha', \lambda')\) be as in Lemma \ref{monlemma} and \(\lambda'=2\lambda\).
	Set \(P = h_{\leq 2}\), so
	\begin{align*}
		\| u - P\|_{B(x, \lambda)} &\leq \|u-h\|_{B(x, \lambda)} + \|h_{>2}\|_{B(x, \lambda)} \\
		&\leq (\lambda/2)^{-m} \|u-h\|_{B(x, 1/2)} + (2\lambda)^{2+\alpha'} \|h_{>2}\|_{B(x, 1/2)} \\
		&\leq (\lambda/2)^{-m} \epsilon + C (2\lambda)^{2+\alpha'}	.	
	\end{align*}
	Take first \(0<\lambda< 1/2\) such that \(C (2\lambda)^{2+\alpha'} < \lambda^{2+\alpha} / 2\) and then \(\epsilon>0\) so that \((\lambda/2)^{-m}\epsilon <\lambda^{2+\alpha}/2\).
\end{proof}

\emph{We take \(\delta\), \(\epsilon_0\) and \(\lambda\) as in the key step Proposition \ref{keystep}} and proceed with the proof of Theorem \ref{mainthm}. So, let
\(\Delta u = f\) on \(B_2\) and fix some \(x \in B_1\). We can assume \(f(x)=0\). Indeed, if \(u' = u - f(x)r_1^2/4\), then \(\Delta u' (x) =0 \) and \(|Du|_{C^{\alpha}(x)} \leq |Du'|_{C^{\alpha}(x)} + (1/4)|f(x)||Dr_1^2|_{C^{\alpha}(x)}\) and we clearly have \(|Dr_1^2|_{C^{\alpha}(x)} \leq C\). Moreover, dividing \(u\) by \(\|u\|_{B(x, 1)} + (1/\delta) |f|_{C^{\alpha}(x)}\), 
we may assume that \(\|u\|_{B(x, 1)} \leq 1\) and \(|f|_{C^{\alpha}(x)} < \delta\).

The H\"older coefficient \(|f|_{C^{\alpha}(x)}\) is given by
\begin{equation*}
\sup_{d(x, y)<1} \frac{|f(y)-f(x)|}{d(x, y)^{\alpha}} .
\end{equation*}
In particular, we are assuming that
\begin{equation*}
\left( \rho^{-m} \int_{B(x, \rho)} |f-f(x)|^2 \right)^{1/2} < \delta \rho^{\alpha} 
\end{equation*} 
for all \(0<\rho<1\).

For \(k=0, 1, \ldots, \infty\), we consider \(B(x, \lambda^k)\). Rescaling to unit size, we have \(\lambda^{-k}B(x, \lambda^k) = B(x_k, 1)\) with \(x_k=\lambda^{-k}x\). By the finitely many bad scales Lemma \ref{scaleslemma}, there is \(N=N(\epsilon_0, \lambda)\) such that all but at most \(N\) of the scales \(\lambda^k\) are good. More precisely, the scale \(\lambda^k\) being good means that, after introducing cuts on some of the cone factors, \(B(x_k, 1) \subset C(Y)\) for some \(C(Y) \in \mathcal{C}\) with \(|x_k| < \epsilon_0\). 

The set of bad scales can be divided into a collection of at most \(n\) disjoint intervals \(I_1 = [a_1, b_1), \ldots, I_n = [a_n, b_n)\) with the following property. If \(k < a_1\), then \(\lambda^k\) is a good scale with model cone \(C(Y_0) = \R^m_{(\beta)}\). If \(b_1 \leq k < a_2\), then \(\lambda^k\) is a good scale whose model cone \(C(Y_1) = \R^m_{(\beta')}\) has replaced a cone factor \(C(S^1_{2\pi\beta_a})\) of the previous model cone, \(C(Y_0)\), with \(\R^2\); and so on. Finally, for \(k \geq b_n \) we have \(C(Y_d) = \R^m\). The set of bad scales depends on the point \(x\), for regular points arbitrary close to the conical set the numbers \(b_n\) get arbitrary large. See Figure \ref{fig:scales}. For singular points, \(I_n = \emptyset\) and the model cone \(C(Y)\) agrees with the tangent cone at \(x\) for \(k\) sufficiently large. We do not have uniform control on how large \(k\) has to be so that \(\lambda^k\) is a good scale. All we shall use is the uniform bound on the number of bad scales, \(\sum_a |I_a| < N\).

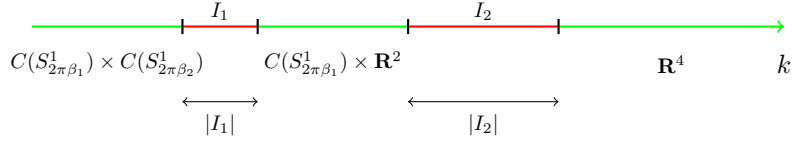
\begin{figure}
	\centering
	\begin{tikzpicture}
	
	\draw[thick, green](-4,0) to (-2,0);
	\draw[thick, red](-2,0) to (-1,0);
	\draw[thick, green](-1, 0) to (1, 0);
	\draw[thick, red](1,0) to (3,0);
	\draw[thick, green, ->](3,0) to (6,0);
	\draw[thick](-2,-.1) to (-2,.1);
	\draw[thick](-1,-.1) to (-1,.1);
	\draw[thick](1,-.1) to (1,.1);
	\draw[thick](3,-.1) to (3,.1);
	\draw[<->](-2, -1) to (-1,-1);
	\draw[<->](1, -1) to (3, -1);

	\draw  (-1.5, -1.3) node (asd) [scale=.8] {\(|I_1|\)};
	\draw  (2, -1.3) node (asd) [scale=.8] {\(|I_2|\)};
	\draw  (-3, -.5) node (asd) [scale=.8] {\(C(S^1_{2\pi\beta_1}) \times C(S^1_{2\pi\beta_2})\)};
	\draw  (0, -.5) node (asd) [scale=.8] {\(C(S^1_{2\pi\beta_1}) \times \mathbf{R}^2\)};
	\draw  (4.5, -.5) node (asd) [scale=.8] {\(\mathbf{R}^4\)};
	\draw  (6, -.5) node (asd) [scale=1] {\(k\)};
	
	\draw (-1.5, .2) node (asd) [scale=.8]
	{\(I_1\)};
	\draw (2, .2) node (asd) [scale=.8]
	{\(I_2\)};
		
	\end{tikzpicture}
	\caption{This picture illustrates the bad scales, in red, for a point \(x=(r_1 e^{i\theta_1}, r_2 e^{i\theta_2})\) lying on the regular locus of \(C(S^1_{2\pi\beta_1}) \times C(S^1_{2\pi\beta_2})\) with \(0<r_1 \leq r_2 \). The model cones $C(Y)$ are indicated for the good scales.}
	\label{fig:scales}
\end{figure}

Set \(u_0 = u\).
Suppose \(\lambda^k\) is a good scale for \(k=0, \ldots, a_1 -1\).
Since \(k=0\) is a good scale, this means that \(|x| < \epsilon_0\). By assumption \(\|u\|_{B(x, 1)} \leq 1\) and \(\|f\|_{B(x, 1)} < \delta\), so we can apply the key step Proposition \ref{keystep} to obtain \(P_1 \in \mathcal{H}_{\leq 2}\) such that \(\|u_0 - P_1\|_{B(x, \lambda)} < \lambda^{2+\alpha} \) and \(\|P_1\|_{B(0,1)} \leq C\).

Set \(u_1 = u_0 - P_1\), so \(\|u_1\|_{B(x, \lambda)} < \lambda^{2+\alpha}\). If \(k=1\) is a good scale, then \(B(x, \lambda) = \lambda B(x_1, 1) \) with \(|x_1| = \lambda^{-1} |x| < \epsilon_0\). Let \(\tilde{u}_1: B(x_1, 1) \to \R\) be given by \(\tilde{u}_1 (\cdot) = \lambda^{-2-\alpha} u_1 (\lambda \cdot) \). We clearly have \(\|\tu_1\|_{B(x_1, 1)}= \lambda^{-2-\alpha} \|u_1\|_{B(x,\lambda)} < 1\). Note that \(\Delta \tilde{u}_1 = \lambda^{-\alpha} f(\lambda \cdot) \), so \(\|\Delta \tilde{u}_1\|_{B(x_1, 1)} = \lambda^{-\alpha} \|f\|_{B(x, \lambda)} < \delta \). We are again in the setting of Proposition \ref{keystep}, so there is \(\tilde{P}_2 \in \mathcal{H}_{\leq 2}\) -with \(\|\tilde{P}_2\|_{B(x_1, 1)} \leq C\)- such that \(\| \tu_1 - \tilde{P}_2\|_{B(x_1, \lambda)} < \lambda^{2+\alpha}\). Set
\begin{align*}
	u_2 &= \lambda^{2+\alpha} (\tu_1 - \tilde{P}_2) (\lambda^{-1} \cdot) \\
	&= u_1 - P_2 ,
\end{align*}
with \(P_2 = \lambda^{2+\alpha} \tilde{P}_2(\lambda^{-1}\cdot)\). Clearly, \(\lambda^{-1}B(x, \lambda^2) = B(x_1, \lambda) \), and we check that 
\begin{align*}
	\|u_2\|_{B(x, \lambda^2)} &= \lambda^{2+\alpha} \| \tu_1 - \tilde{P}_2\|_{B(x_1, \lambda)} \\ &< \lambda^{2(2+\alpha)} . 
\end{align*}
We proceed the same way. We obtain \(P_1, \ldots, P_k \in \mathcal{H}_{\leq 2}\) such that
\(u_k = u - (P_1 + \ldots +P_k)\) satisfies \(\|u_k\|_{B(x, \lambda^k)} < \lambda^{(2+\alpha)k}\) for \(k=0, \ldots, a_1\).

Now, for \(k \in I_1 = [a_1, b_1)\) we have bad scales, and we set \(P_k = 0\), that is \(u_{k+1} = u_k\). Clearly, \(\|u_{k+1}\|_{B(x, \lambda^{k+1})} \leq \lambda^{-m/2} \|u_k\|_{B(x, \lambda^k)}\) whenever \(u_{k+1} = u_k\). If we let \(Q_1 = (\lambda^{-m/2} \lambda^{-(2+\alpha)})^{|I_1|}\),  we can then guarantee that  \(\|u_k\|_{B(x, \lambda^k)} \leq Q_1 \lambda^{(2+\alpha)k}\) for all \(k=0, \ldots, b_1\).

Consider now \(k= b_1\), by the previous step, we have \(\|u_k\|_{B(x, \lambda^k)} < Q_1 \lambda^{(2+\alpha)k}\). Let
\[\tu_k = Q_1^{-1} \lambda^{-(2+\alpha)k} u_k(\lambda^k \cdot) ,\]
so \(\| \tu\|_{B(x_k, 1)} = Q^{-1} \lambda^{-(2+\alpha)k} \|u_k\|_{B(x, \lambda^k)} < 1\); and \(\Delta \tu_k = Q_1^{-1} \lambda^{-\alpha k} f(\lambda^k \cdot)\), so \(\|\Delta \tu_k\|_{B(x_k, 1)} < Q_1^{-1} \delta< \delta\). Since \(k=b_1\) is a good scale, we can identify -after doing a cut and translating- \(B(x_k, 1) \subset C(Y_1)\) with \(|x_k| < \epsilon_0\). By Proposition \ref{keystep} we have \(\tilde{P}_{k+1} \in \mathcal{H}_{\leq 2} (C(Y_1))\) such that \(\|\tu_k - \tilde{P}_{k+1}\|_{B(x_k, \lambda)} < \lambda^{2+\alpha}\). Set
\begin{align*}
	u_{k+1} &= Q_1 \lambda^{(2+\alpha)k} (\tu_k - \tilde{P}_{k+1}) (\lambda^k \cdot) \\
	&= u_k - P_{k+1} ,
\end{align*}
with \(P_{k+1} = Q_1 \lambda^{(2+\alpha)k} \tilde{P}_{k+1} (\lambda^{-k} \cdot)\). We check that
\begin{align*}
	\|u_{k+1}\|_{B(x, \lambda^{k+1})} &= Q_1 \lambda^{(2+\alpha)k} \|\tu - \tilde{P}_{k+1}\|_{B(x_k, \lambda)} \\
	&< Q_1 \lambda^{(2+\alpha)(k+1)} .
\end{align*}
We proceed the same in the good range of scales \(k=b_1, \ldots, a_2-1 \). We have constructed \(u_k = u - (P_1 + \ldots +P_k)\) that satisfy \(\|u_k\|_{B(x, \lambda^k)} < Q_1 \lambda^{(2+\alpha)k}\) for \(k=0, \ldots, a_2\).

When \(k=a_2\), we hit again a range of bad scales \(I_2 = [a_2, b_2)\). Same as before, for \(k= a_2, \ldots, b_2-1\), we set \(u_{k+1} = u_k\). We let \(Q_2 = (\lambda^{-m/2} \lambda^{-(2+\alpha)})^{|I_2|}\), so  \(\|u_k\|_{B(x, \lambda^k)} < Q_1 Q_2 \lambda^{(2+\alpha)k}\) for \(k=0, \ldots, b_2\). For the good range \(k=b_2, \ldots, a_3-1\) we set 
\[\tu_k = Q_1^{-1}Q_2^{-1} \lambda^{-(2+\alpha)k} u_k(\lambda^k \cdot) .\]
We use Proposition \ref{keystep} to obtain \(\tilde{P}_{k+1} \in \mathcal{H}_{\leq 2}(C(Y_2))\) such that, if we set \(P_{k+1} = Q_1 Q_2\lambda^{(2+\alpha)k} \tilde{P}_{k+1} (\lambda^{-k} \cdot)\) and \(u_{k+1} = u_k - P_{k+1}\), then \(\|u_k\|_{B(x, \lambda^k)} < Q_1Q_2 \lambda^{(2+\alpha)k}\) for all \(k=0, \ldots, a_3\). We proceed the same way, after \(I_n\) there are no more bad scales, so all \(k \geq b_n\) are good.
As a result, we get a sequence \(u_k = u - (P_1 + \ldots + P_k)\) with
\begin{equation*}
	\|u_k\|_{B(x, \lambda^k)} \leq C \lambda^{(2+\alpha)k}
\end{equation*}
for all \(k=0, 1, \ldots, \infty\). Here, \(C = (\lambda^{-m/2} \lambda^{-(2+\alpha)})^{N}\) with \(\sum_a |I_a| \leq N=N(\epsilon_0, \lambda)\) given by Lemma \ref{scaleslemma}.

Let \(\tau_k = DP_k (x)\). Because \(D\) is homogeneous of degree two, we have \(D\tilde{P}(\lambda^{-k} \cdot) = \lambda^{-2k} (D\tilde{P}_k)(\lambda^{-k} \cdot)\). Up to a constant factor, \(P_k\) is equal to \(\lambda^{(2+\alpha)k} \tilde{P}_k (\lambda^{-k}\cdot)\). Therefore, \(\tau_k = \lambda^{k\alpha} D\tilde{P}_k(x_k)\). One detail is that we have uniform bounds \(|D'\tilde{P}_k|(x_k) \leq C\) for all second order derivatives \(D'\) of the corresponding model cone, after the cut and translation identification is done. The model cones are less singular than the original, and we get estimates on more derivatives; in particular we can express \(D\) as a combination of operators \(D'\) with uniformly bounded (smooth) coefficients. The Euclidean directions in \(D\) will remain unchanged, but the cone directions might change into Euclidean. We are performing the cuts uniformly away from the tips of the cones. The basic fact we need is that the vector fields \(\p_r\) and \(r^{-1}\p_{\theta}\) on \(\R^2 \setminus \{0\}\) can be written as a combination of the Cartesian vector fields \(\p_{s_1}, \p_{s_2}\) with uniformly bounded (smooth) coefficients on discs of radius \(1\) which are at distance \(\geq 2\), say, from the origin. The outcome is that
\begin{equation*}
	|\tau_k| \leq C \lambda^{k\alpha} ,
\end{equation*}
so the series \(\sum_k \tau_k\) converges and we set \(\tau = \sum_{k=0}^{\infty} \tau_k\).

By Lemma \ref{auxlem}, in order to prove the estimate \eqref{mainestimate}, it is enough to show that
\begin{equation} \label{mainest2}
	\|Du - (\tau_1 + \ldots +\tau_k)\|_{B(x, \lambda^k)} \leq C \lambda^{k\alpha} .
\end{equation}
The scaled version of Proposition \ref{L2bound} gives us 

\begin{equation}\label{scaledL2}
	\begin{aligned} 
	\|Du_k\|_{B(x, \lambda^k)} &\leq C (\|f\|_{B_{k-1}} + \lambda^{-2k}\|u_k\|_{B_{k-1}}) \\
	&\leq C \lambda^{k\alpha} .
	\end{aligned}
\end{equation}

We rewrite \(Du_k\) as follows
\begin{equation}\label{Duk}
	\begin{aligned} 
	Du_k &= Du - (DP_1 + \ldots + DP_k) \\
	&= Du - (\tau_1 + \ldots + \tau_k) - (DP_1 - DP_1(x) + \ldots +DP_k - DP_k(x)) .
	\end{aligned}
\end{equation}

Recall that \(D' \tilde{P}_j =\) constant for all second derivatives \(D' \in \mathcal{D}_2\) of the corresponding model cone. We use again that \(D\) can be expressed as a combination of derivatives \(D'\) with uniform control on the H\"older norm of the coefficients. More precisely, the only case in which \(DP_j\) is not constant is when \(D\) contains a conical derivative, say \(\p_{r_a}\), and \(r_a(x) \neq 0\). Then for all \(j \geq j_0\), such that \(\lambda^{-j_0}r_a(x) > c_{\beta}\) we introduce a cut in the \(C(S^1_{2\pi\beta_a})\) factor. The coefficients of \(\p_r\) with respect to the Cartesian coordinate vector fields, in a disc of radius \(1\) which is at distance \(R\) from the origin, are constant up to an \(O(1/R)\) error. We conclude that the H\"older semi-norm of the coefficients of \(D\) in terms of \(D'\) -and hence the H\"older semi-norm of \(D\tilde{P}_j\)- is bounded by a uniform constant multiple of \(\lambda^j r_a(x)\). Therefore,
\begin{align*}
	\sum_{j=1}^{k} \|DP_j - DP_j(x)\|_{B(x, \lambda^k)} &= \sum_j  \lambda^{\alpha j} \|D\tilde{P}_j - D\tilde{P}_j(x_j)\|_{B(x_j, \lambda^{k-j})} \\
	&\leq C \lambda^{\alpha k} \sum_{j=j_0}^{\infty} \lambda^j r_a(x)^{-1} . 
\end{align*}
Now, \(\sum_{j=j_0}^{\infty} \lambda^j r_a(x)^{-1} = \lambda^{j_0} r_a(x)^{-1} \sum_{j=0}^{\infty} \lambda^j \) and \(\lambda^{j_0}r_a(x)^{-1} < 1/c_{\beta}\). So, 
\begin{equation} \label{calphapol}
	\sum_{j=1}^{k} \|DP_j - DP_j(x)\|_{B(x, \lambda^k)} \leq C \lambda^{k\alpha} .
\end{equation}

The main estimate \eqref{mainest2} follows from equations \eqref{scaledL2}, \eqref{Duk} and \eqref{calphapol}.

\begin{remark}[Vanishing of conical derivatives at the singular locus]
	Note that if \(x\) lies on the singular set, say \(r_a(x) =0\), and if \(D\) is a mixed Euclidean-conical or conical-conical derivative that contains either \(\p_{r_a}\) or \(r_a^{-1}\p_{\theta_a}\); then \(DP=0\) for all \(P \in \mathcal{H}_{\leq 2}\) involved in the proof. This means that \(\tau_k = 0\) for all \(k\), so \(\tau=0\) and \(Du(x)=0\).
\end{remark}

\begin{remark}[Lower order estimates]
	The same method of proof of Theorem \ref{mainthm}, but subtracting the sub-linear harmonic parts, leads to first order estimates. Indeed, even weaker assumptions on the regularity of \(f\), lead to a \(C^{\alpha}\)-bound of the  gradient components:
	\begin{equation*}
	\frac{\p u}{\p r_a}, \hspace{2mm} \frac{1}{r_a}\frac{\p u}{\p \theta_a}, \hspace{2mm} \frac{\p u}{\p s_i} .
	\end{equation*}
	Again, the sublinear harmonic functions only depend on the Euclidean variables, so the normal component of \(\nabla u\) along \(\{r_a=0\}\) must vanish. 
	
	The vanishing condition admits the following geometric explanation. The holonomy of \(g\) along a loop that goes around \(\{r_a = 0\}\) is non-trivial on the \(C(S^1_{2\pi\beta_a})\) factor. Take small loops around \(\{r_a=0\}\) that shrink to a point and parallel transport \(\nabla u\). If \(|\nabla u|\) and \(|\p_{\theta_a} \nabla u|\) are uniformly bounded, then the gradient must become invariant under parallel transport as the loops  shrink. The only way this can happen is if the \(C(S^1_{2\pi\beta_a})\)-component of \(\nabla u\) vanishes. 
\end{remark}

\subsection{K\"ahler metrics}
Assume \(m\) is even. Introducing complex coordinates \(z_a = (1/\beta_a)r_a e^{i\theta_a}\) on each cone factor, identifies  \(\R^m_{(\beta)}\) with \(\C^{m/2}\) endowed with the singular K\"ahler metric
\[\omega_{(\beta)} = \sum_a |z_a|^{2\beta_a-2} idz_ad\bar{z}_a + \omega_{\C^{m/2-n}} . \]
Theorem \ref{mainthm} provides a bound for \(|\dd u|_{\alpha}\), in the sense that the components of \(\dd u\)  with respect to an \(\omega_{(\beta)}\)-orthonormal co-frame of \((1,1)\)-forms are \(C^{\alpha}\). Given a \(C^{\alpha}\) closed \((1,1)\)-form \(\eta\) such that \(\omega := \omega_{(\beta)} + \eta >0 \); it is possible to extend Theorem \ref{mainthm} to \(\Delta_{\omega}\). This is done, by means of the freezing coefficients method, in \cite[Section 3.5]{GuoSongII}; see also \cite[Section 4.2]{Donaldson} and \cite[Section 2]{Brendle}.

\appendix

\section{The two dimensional cone}

\subsection{Geometry of balls} \label{2cone}

We recall here some basic facts on the geometry of the two-dimensional cone \(C(S^1_{{2\pi\beta}})\). In polar coordinates on \(\R^2\), the metric is
\[g = dr^2 + \beta^2 r^2 d\theta^2  \]
and \(r\) measures the intrinsic distance to the apex, located at \(0\).

If we cut \(C(S^1_{2\pi\beta})\) along a geodesic ray starting from the vertex, say \(\theta =\) constant, we get a wedge in the Euclidean plane of angle \(2\pi\beta\). Given a point \(p \in C(S^1_{2\pi\beta}) \setminus \{o\}\), we can always cut along the `opposite ray' to \(p\). More precisely, \(p\) lies on a circle of length \(2\pi\beta r\) given by the points which are at distance \(r=r(p)\) from the vertex. On that circle, there is a unique antipodal point \(q\) which is at maximal distance \(\pi \beta r\) from \(p\). We cut along the ray starting from \(o\) and going through \(q\),  so then \(p\) will lie on the `middle of the wedge', which means equidistant from the two sides of the wedge. In coordinates, this means that we choose \(\theta \in (-\pi, \pi)\) with \(\theta(p)=0\) and we set \(\tilde{\theta}=\beta\theta\) to get the Euclidean wedge \(-\pi\beta<\tilde{\theta}<\pi\beta\).

We look at geodesic balls \(B(p, \rho) \subset C(S^1_{2\pi\beta}) \). The situation is pretty clear when \(p\) is the vertex. All balls centered at \(o\) are isometric after rescaling to unit size, that is \(\rho^{-1}B(o, \rho)=B(o, 1)\). We consider \(p \in C(S^1_{2\pi\beta})\setminus \{o\}\). It is clear that \(\rho^{-1}B(p, \rho)\) will be isometric to the Euclidean unit disc if \(\rho\) is small;  and for large values of \(\rho\) it will be isometric to a unit ball in \(C(S^1_{2\pi\beta})\) with its center very close to the origin. We can quantify this transition as follows. First, it is enough to consider unit balls  \( B(\tilde{p}, 1) \subset C(S^1_{2\pi\beta})\) because
\(\rho^{-1}B(p, \rho) = B(\rho^{-1}p, 1)\) where \(\rho^{-1}(r, \theta)=(\rho^{-1}r, \theta)\). Taking a cut, it is clear from Figure \ref{fig:balls} that \(B(\tilde{p}, 1)\) is isometric to the Euclidean unit disc if \(d(\tilde{p}, o)>c_{\beta}\), where 
\[ c_{\beta} = \begin{cases}
\sin (\pi \beta) \hspace{2mm} \mbox{if} \hspace{2mm} 0 < \beta \leq 1/2 \\
1 \hspace{2mm} \mbox{if} \hspace{2mm} \beta > 1/2 .
\end{cases}\]

Given \(\epsilon >0\), we say that \(B(\tilde{p}, 1)\) is \(\epsilon\)-close to \(B(o, 1)\), if \(d(\tilde{p}, o)< \epsilon\). The conclusion is that, if we let \(p \in C(S^1_{2\pi\beta}) \setminus \{o\}\) and \(r=r(p)\), then
\[\rho^{-1} B(p, \rho) = \begin{cases}
B_1^{E} \subset \R^2 \hspace{2mm} \mbox{if} \hspace{2mm} \rho < c_{\beta} r \\
B(\tilde{p}, 1) \hspace{2mm} \mbox{with } d(\tilde{p}, o) < \epsilon \hspace{1mm}  \mbox{if} \hspace{2mm} \rho > \epsilon^{-1} r
\end{cases} \]
We take logarithms, that is we introduce a scale parameter \(s\) and write \(\rho=e^s\). We say that a scale \(s\) is good if \(\rho^{-1}B(p, \rho)\) is isometric either to the Euclidean unit disc \(B^E_1\) or to \(B(\tilde{p}, \epsilon)\) with \(d(\tilde{p, o})<\epsilon\). By the above discussion, if either \(s<\log c_{\beta} + \log r\) or \(s > \log(\epsilon^{-1}) + \log r\) then \(s\) is a good scale. In other words, the set of bad scales has size uniformly bounded above, independently of \(p\), by \(\log(\epsilon^{-1}) + \log (c_\beta^{-1})\). The point is that, while we do not have uniform control over \(p\) of which scales are good or bad, we do have a uniform bound on the set of bad scales.

\begin{figure}
	\centering
	\begin{tikzpicture}[scale=.8]
	
	\draw[] (-4,0) to (0, 2);
	\draw[] (-4,0) to (0,-2);
	\draw[dotted] (-4,0) to (0,0);
	\draw[dashed] (-1, 1.5) to (-.25, 0);
	\draw[dashed] (-1, -1.5) to (-.25, 0);
	
	\draw[] (6,0) to (3, 1);
	\draw[] (6, 0) to (3, -1);
	\draw[dashed] (6, 0) to (8,0);
	\draw[dotted] (8, 0) to (10, 0);
	
	\draw[blue] (-.25, 0) circle [radius=1.677] ;
	
	\draw[blue] (8, 0) circle [radius=2] ; 
	
	\centerarc[<->](-4,0)(30:330:2) ;
	\centerarc[](-4,0)(0:25:1) ;
	
	\centerarc[<->](6,0)(165:195:2) ;
	\centerarc[](6,0)(0:160:.5) ;
		
	\draw[] (-2.7,.2) node {\(\pi\beta\)} ;
	\draw[] (-1.5, .2) node[scale=.9] {\(r\)} ;
	\draw[] (.25, .8) node[scale=.9] {\(r \sin(\pi\beta)\)} ;
	
	\draw[] (6.6,.55) node[scale=.9] {\(\pi\beta\)} ;
	\draw[] (7, -.25) node {\(r\)} ;
	
	\filldraw[] (-.25, 0) circle[radius=.05, fill=black];
	\draw[] (0, -.2) node {\(p\)} ;
	
	\filldraw[] (8, 0) circle[radius=.05, fill=black];
	\draw[] (8.2, -.3) node {\(p\)} ;

	\end{tikzpicture}
	\caption{Difference between \(\beta<1/2\), on the left, against \(\beta>1/2\), on the right. The dotted line on the left is a geodesic loop, its homotopy class  in \(B(p, \rho)\) is non trivial for \(c_{\beta}r < \rho < r\).}
	\label{fig:balls}
\end{figure}
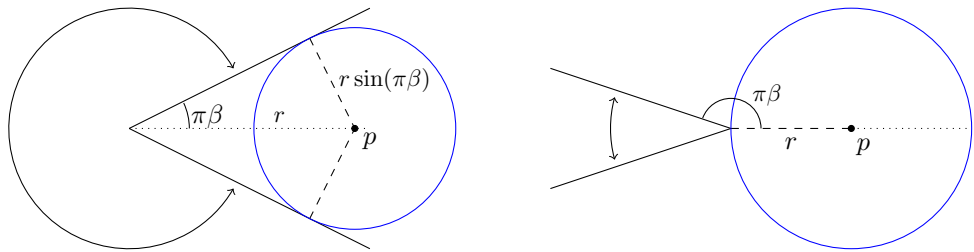

Note that if \(0 < \beta < 1/2\) and \( c_{\beta} r < \rho < r\), then \(B(p, \rho)\) is homeomorphic to an annulus. If \(\epsilon< c_{\beta}\) then these values of \(\rho\), for which \(B(p, \rho)\) is homeomorphic to an annulus, are comprised among the bad scales.

\subsection{Bounded harmonic functions}

In polar coordinates, the Laplacian on the cone \(C(S^1_{2\pi\beta})\) is given by
\[\Delta = \frac{\p^2}{\p r^2} + \frac{1}{r} \frac{\p}{\p r} + \frac{1}{\beta^2 r^2} \frac{\p^2}{ \p\theta^2} .\]
If we change to Cartesian coordinates \(u=r\cos\theta\), \(v=r\sin\theta\), then
\[ g = (\cos^2\theta + \beta^2 \sin^2 \theta) du^2 + (1-\beta^2) \sin\theta\cos\theta \left(dudv + dvdu\right) + (\beta^2\cos^2\theta +  \sin^2 \theta) dv^2 . \]
The metric coefficients \(g_{ij}\) are uniformly bounded but are not continuous at \(0\). By the much more general De Giorgi-Nash-Moser theory, we know that weak harmonic functions are automatically H\"older continuous. This is a non-trivial statement, functions in \(W^{1,2}\) can be unbounded, the prototypical example is \(\log |\log r|\) in a neighborhood of the origin of \(\R^2\).  We provide  a self-contained proof of the H\"older continuity  in the two-dimensional case. 

\begin{lemma}
	Let \(u \in W^{1, 2}(B_1)\) satisfy \(\Delta u =0\), then \(u\) is H\"older continuous.
\end{lemma}
	
\begin{proof}
The following applies, more generally, to the singular metrics \(\R^m_{(\beta)}\) considered in this article.	
We use the test function \(\psi= \eta^2 u\) as in Lemma \ref{Caccioppoli}, with \(\eta =1\) on \(B_{1/2}\) and \(\mbox{supp}(\eta) \subset B_1\). We get
	\begin{equation*}
	\int_{B_1} \eta^2 |\nabla u|^2 \leq 4 \int_{B_1}  |\nabla \eta|^2 u^2 .
	\end{equation*}
	Let \(A = B_1 \setminus B_{1/2}\) and let \(u_A\) be the average of \(u\) over \(A\). Replacing \(u\) by \(u-u_A\) and using the Poincar\'e inequality, we have
	\begin{align*}
	\int_{B_{1/2}} |\nabla u|^2 &\leq C \int_A (u-u_A)^2 \\
	&\leq  C \int_A |\nabla u|^2 \\
	&= C \left(\int_{B_1} |\nabla u|^2 - \int_{B_{1/2}} |\nabla u|^2 \right) .
	\end{align*}
	We get that
	\[  \int_{B_{1/2}} |\nabla u|^2 \leq \frac{C}{C+1}  \int_{B_{1}} |\nabla u|^2  \]
	and we can iterate to obtain
	\[  \int_{B_{2^{-k}}} |\nabla u|^2 \leq \left(\frac{C}{C + 1}\right)^k  \int_{B_{1}} |\nabla u|^2 . \]
	Equivalently, \(\int_{B_r} |\nabla u|^2 = O (r^{\mu})\) as \(r \to 0\). Here, \(\mu>0\) is given by \(2^{-\mu} = (C+1)^{-1} C\). We can use Poincar\'e again,  \(\int_{B_r} (u-u_{B_r})^2 \leq C r^{2} \int_{B_r} |\nabla u|^2 \), therefore
	\[\|u - u_{B_r}\|_{B_r} \leq C r^{(2+\mu - m)/2} .\]
	In particular, this shows that if \(m=2\) then \(u\) is H\"older continuous.
\end{proof}

The eigenfunctions of the Laplace operator \(\beta^{-2}\p^2_{\theta}\) of \(S^1_{2\pi\beta}\) are \(\cos (k\theta)\), \(\sin(k\theta)\), with eigenvalues \(k^2\beta^2\) for \(k\geq 0\). The corresponding homogeneous harmonic functions are
\[r^{k/\beta}\cos(k \theta), \hspace{2mm} r^{k/\beta}\sin(k\theta).\]
Every harmonic function \(u\) on \(B_1\) admits an expansion
\begin{equation} \label{taylor}
	u = \sum_k \left(a_k r^{k/\beta}\cos(k \theta) + b_k r^{k/\beta}\sin(k\theta)\right) .
\end{equation}

We introduce a complex variable, \(z = (1/\beta) r^{1/\beta} e^{i\theta}\); note that the change of coordinates is singular at the origin. In this complex coordinate, we have
\[ dr^2 + \beta^2 r^2 d\theta^2 = |z|^{2\beta-2} |dz|^2 .\]
and \(4\Delta = |z|^{2-2\beta} \p^2/\p z
\p \bar{z}\). Since the weak harmonic function \(u\) is continuous, by the standard removability of singularities, \(u\) is a smooth harmonic function of the complex variable \(z\) and the expansion \eqref{taylor} is the usual Taylor series.

\bibliographystyle{amsalpha}
\bibliography{bibprodcones}

\end{document}